\documentclass[10pt]{article}

\usepackage{amsmath}    
\usepackage{graphicx}   
\usepackage{verbatim}   
\usepackage{color}      
\usepackage{hyperref}   

\usepackage{amssymb}
\usepackage{amsthm}

\theoremstyle{theorem}
\newtheorem{Def}{Definition}[section]

\newtheorem{proposition}[Def]{Proposition}
\newtheorem{lemma}[Def]{Lemma}
\newtheorem{theorem}[Def]{Theorem}
\newtheorem{corollary}[Def]{Corollary}
\newtheorem{remark}[Def]{Remark}

\newcommand{\mf}{\mathcal{F}}

\newcommand{\br}{\mathbb{R}}
\newcommand{\bN}{\mathbb{N}}
\newcommand{\pr}{\mathbb{P}}
\newcommand{\bQ}{\mathbb{Q}}
\newcommand{\bE}{\mathbb{E}}

\newcommand{\sm}{\sigma}
\newcommand{\eps}{\epsilon}
\newcommand{\1}{{\bf 1}}

\allowdisplaybreaks

\begin{document}

\title{Approximation for non-smooth functionals of stochastic differential equations  with irregular drift 
}


\author
{
Hoang-Long Ngo\footnote{Hanoi National University of Education, 136 Xuan Thuy - Cau Giay - Hanoi - Vietnam, email: $\qquad$ngolong@hnue.edu.vn}
$\quad $ and $\quad$ 
Dai Taguchi\footnote{Osaka University, 1-3, Machikaneyama-cho, Toyonaka, Osaka 560-8531, Japan, email: dai.taguchi.dai@gmail.com}
}

\date{}
\maketitle

\begin{abstract}
This paper aims at developing a systematic study for the weak rate of convergence of the Euler-Maruyama scheme for stochastic differential equations with very irregular drift and constant diffusion coefficients.
We apply our method to obtain the rates of approximation for the expectation of various non-smooth functionals of both stochastic differential equations  and killed diffusion.
We also apply our method to the study of the weak approximation of reflected stochastic differential equations whose drift is H\"older continuous. 
\\
\textbf{2010 Mathematics Subject Classification}: 60H35, 65C05, 65C30\\\\
\textbf{Keywords}: Euler-Maruyama approximation, Irregular drift, Monte Carlo method, Reflected stochastic differential equation, Weak approximation

\end{abstract}


\section{Introduction} \label{sec:intro}

Let $T>0$ be fixed and $(X_t)_{0 \leq t \leq T}$ be the solution to 
$$dX_t = b(X_t) dt + \sm(X_t) dW_t, \quad X_0 = x_0 \in \br^d, \  0 \leq t \leq T,$$
where $W$ is a $d$-dimensional Brownian motion. The diffusion $(X_t)_{0 \leq t \leq T}$ is used to model many random dynamical phenomena in many fields of applications.  In practice, one often encounters the problem of evaluating functionals of the type $\bE[f(X)]$ for some given function $f:C[0,T] \to \mathbb{R}$. For example, in mathematical finance the function $f$ is commonly referred as a {\it payoff} function. Since they are rarely analytically tractable, these expectations are usually approximated using numerical schemes. One of the most popular approximation methods  is the Monte Carlo Euler-Maruyama method which consists of two steps: 
\begin{enumerate}
\item The diffusion process $(X_t)_{0 \leq t \leq T}$ is approximated using the Euler-Maruyama scheme $(X^h_t)_{0 \leq t \leq T}$ with a small time step $h>0$:
\begin{equation} \label{def:EulerX}
d X^h_t = b(X^h_{\eta_h(t)}) dt + \sm(X^h_{\eta_h(t)})dW_t, \quad X^h_0 = x_0, \ \eta_h(t) = kh, 
\end{equation}
for   $t\in [kh, (k+1)h)$, $k\in\mathbb{N}$.
\item The expectation $\bE[f(X)]$ is approximated using $\frac{1}{N} \sum_{i=1}^N f(X^{h,i})$ where $(X^{h,i})_{1\leq i \leq N}$ are $N$ independent copies of $X^h$. 
\end{enumerate}
This approximation procedure is influenced by two sources of errors: a discretization error and a statistical error
$$Err(f,h)  := Err(h) := \bE [f(X)]  -\bE [f(X^h)], \quad \text{and } \ \bE [f(X^h)] - \frac{1}{N}\sum_{i=1}^N f(X^{h,i}).$$ 
We say that the Euler-Maruyama approximation $(X^h)$ is of \emph{weak order} $\kappa >0$ for a class $\mathcal{H}$ of functions $f$ if there exists a constant $K(T)$ such that for any $f\in \mathcal{H}$,
$$| Err(f,h)| \leq K(T) h^\kappa.$$

The effect of the statistical error can be handled  by the classical central limit theorem or large deviation theory. Roughly speaking, if $f(X^h_T)$ has a bounded variance, the $L^2$-norm of the statistical error is bounded by $N^{-1/2}Var(X^h_T)^{1/2}$. Hence, if the Euler-Maruyama approximation is of weak order  $\kappa$, the optimal choice of the number of Monte Carlo iterations  should be $N= O(h^{-2\kappa })$ in order to minimize the computational cost. Therefore, it is of both theoretical and practical importance to understand the weak order of  the Euler-Maruyama approximation.

It has been shown that under sufficient regularity on the coefficients $b$ and $\sm$ as well as $f$, the weak order of the Euler-Maruyama approximation is $1$. This fact is proven by writing the discretization error $Err(f,h)$ as a sum of terms involving the solution of  a parabolic partial differential equation (see \cite{BT, TT, Mi, KP, G08}).   It should be noted here that besides the Monte Carlo Euler-Maruyama method, there are many other related approximation schemes for  $\bE[f(X_T)]$ which have either higher weak order or lower computational cost. For example, one can use Romberg extrapolation technique to obtain very high weak order as long as $Err(h)$ can be expanded in terms of powers of $h$ (see \cite{TT}). When $f$ is a Lipschitz function and the strong rate of approximation is known, one can implement a Multi-level Monte Carlo simulation which can significantly reduce the computation cost of approximating $\bE[f(X)]$ in many cases (see \cite{G}). It is also worth looking at some algebraic schemes introduced in \cite{Ku}. However, all the accelerated schemes mentioned above require sufficient regularity conditions on the coefficients $b, \sm$ and  the test function $f$.

The stochastic differential equations with non-smooth drift appear in many applications, especially when one wants to model sudden changes in the trend of a certain random dynamical phenomenon (see e.g., \cite{KP}). There are many papers studying the Euler-Maruyama approximations in this context. In \cite{GK} (see also \cite{CS}), it is shown that when the drift is only measurable, the diffusion coefficient is non-degenerate and  Lipschitz continuous then the  Euler-Maruyama approximations converges to the solution of stochastic differential equation. 
The weak order of the Euler-Maruyama scheme when both coefficients $b$ and $\sm$ as well as payoff functions $f$ are H\"older continuous has been studied in \cite{KP, MP}. 
In the papers \cite{KLY} and \cite{NT, NT_IMA}, the authors studied the weak and strong convergent rates of the Euler-Maruyama scheme for specific classes of stochastic differential equations with discontinuous drift. We would like to note here some very recent papers on the unbiased simulation \cite{HTT} and exact simulation \cite{BK, DMR, EM, PRT} methods for stochastic differential equations with irregular coefficients.

The aim of the present paper is to investigate  the weak order of the Euler-Maruyama approximation for stochastic differential equations whose diffusion coefficient $\sigma$ is constant,  whereas the drift coefficient $b$ may have a very low regularity, or could even be  discontinuous. More precisely, we consider a class $\mathcal{A}$  of functions which contains not only smooth functions but also some discontinuous one such as indicator function.
The drift $b$ will then be assumed to be either in  $\mathcal{A}$ or $\alpha$-H\"older continuous.   It should be noted that no smoothness assumption on the payoff function $f$ is needed in our framework. As a by product of our method, we  establish the weak order of the Euler-Maruyama approximation for some particular functionals $f$ which include
the path-wise maximum of the diffusion, integral of diffusion with respect to time as well as the approximation of 
a diffusion processes killed when it leaves an open set. We also apply our method to study the weak approximation of reflected stochastic differential equation whose drift is H\"older continuous.

The remainder of this paper is organized as follows. In the next section we introduce some notations and assumptions for our framework together with the main results. All proofs are deferred to Section \ref{sec:proofs}.

\section{Main Results} \label{sec:main}
\subsection{Notations}

For an invertible $d\times d$-matrix $A=(A_{i,j})_{1\leq i,j\leq d}$, we define for all $x,y \in \br^d$,
\begin{align}\label{nomal}
g_{A}(x,y):=\frac{\exp\left(-\frac{1}{2}\langle A^{-1}(y-x),y-x \rangle \right)}{(2\pi)^{d/2} \sqrt{\det{A}}}.
\end{align}
In particular we denote $g_c(x,y)=g_{cI}(x,y)$ for $c \in \br$ where the matrix $I$ is the identity matrix.

A function $\zeta: \br^d \to \br$ is called {\it exponentially bounded} or {\it polynomially bounded} if there exist positive constants $K, p$ such that  $|\zeta (x)| \leq K e^{K|x|}$ or  $|\zeta (x)| \leq K(1+|x|^p)$, respectively.

Let $\mathcal{A}$ be a class of exponentially bounded functions $\zeta : \br^d \to \br$ such that there exists a sequence of functions $(\zeta _N)_{N \in \bN} \subset C^1(\br^d)$ satisfying:
\begin{equation} \label{repX}
\begin{cases}
\mathcal{A}(i): & \zeta _N \to \zeta  \text{ in } L^1_{loc}(\br^d), \\
\mathcal{A}(ii): &\sup_N |\zeta _N(x)| +|\zeta (x)|\leq K e^{K|x|}, \text{for all } x \in \br^d\\ 
\mathcal{A}(iii): & \sup_{N,u>0; \ a \in \br^d} e^{-K|a|-Ku}\int_{\br^d}  |\nabla \zeta _N(x+a)| \frac{e^{-|x|^2/u}}{u^{(d-1)/2}} dx < K, 
\end{cases}
\end{equation}
for some positive constant $K$. We call $(\zeta_N)_{N \in \mathbb{N}}$ an approximation sequence of $\zeta$ in $\mathcal{A}$.

The following propositions shows that this class  is quite large. 
\begin{proposition} \label{prop:A}
i) If $\xi, \zeta  \in \mathcal{A}$ then $\xi\zeta  \in \mathcal{A}$ and $a_1 \xi  + a_2 \zeta  \in \mathcal{A}$ for any $a_1, a_2 \in \br$. \\
ii) Suppose that $A$ is a non-singular $d \times d$-matrix, $B \in \br^d$. Then $\zeta  \in \mathcal{A}$ iff $\xi(x) := \zeta (Ax + B) \in \mathcal{A}$, for all $x \in \br^d$.
\end{proposition}
It is easy to verify that the class $\mathcal{A}$ contains all $C^1(\br^d)$ functions which has all first order derivatives polynomially bounded. Furthermore, the class $\mathcal{A}$ contains also some non-smooth  functions of the type $\zeta (x) = (x_1 - a)^+$ or $\zeta (x) = I_{a < x < b}$ for some $a, b \in \br^d$.
Moreover, we call a function $\zeta : \br^d \to \br$  \emph{monotone in each variable separately} if for each $i=1, \ldots, d$, the map $x_i \mapsto \zeta (x_1, \ldots, x_i, \ldots, x_n)$ is monotone for all $x_1, \ldots, x_{i-1},  x_{i+1}, \ldots, x_d \in \br$.
\begin{proposition} \label{classA}
Class  $\mathcal{A}$ contains all exponentially bounded functions which are monotone in each variable separately.
\end{proposition}
The proofs of Propositions \ref{prop:A} and \ref{classA} and further properties of class $\mathcal{A}$ were presented in \cite{KMN} and \cite{NT}.

We recall that a function $\zeta: \mathbb{R}^d \to \mathbb{R}^d$ is called $\alpha$-H\"older continuous for some $\alpha \in (0,1]$ if there exists a positive constant $C$ such that $|\zeta(x) - \zeta(y)| \leq C|x-y|^\alpha$ for all $x,y \in \mathbb{R}^d$.
We denote $\mathcal{M}(\br^d,\br^d)$ the class of all measurable functions $b:\br^d \to \br^d$.
Let $\alpha \in (0,1]$ and let $\mathcal{B}(\alpha)$ be the class fo functions given by
\begin{align}\label{classB}
	\mathcal{B}(\alpha)
	=\left\{
	b \in \mathcal{M}(\br^d,\br^d);~b=b^{H}+b^{A} \text{ where } b^H \text{ is $\alpha$-H\"older and } b_j^A \in \mathcal{A},~j=1,\ldots,d
	\right\}.
\end{align}
A function $b \in \mathcal{B}(\alpha)$ could be very irregular since its trajectory may both change continuously very fast because of the H\"older part $b^H$, and  be  discontinuous at some points because of the part $b^A \in \mathcal{A}$. Moreover, the class $\mathcal{A}$ does not contain any class of H\"older continuous functions $H^\beta$ for any $\beta \in (0,1)$ (see \cite{NT_IMA}).

\subsection{Weak approximation of stochastic differential equations}  \label{subs:sde}

Let $(\Omega, \mf, (\mf_t)_{t \geq 0}, \pr)$ be a filtered probability space and $(W_t)_{t \geq 0}$ be a $d$-dimensional standard Brownian motion. We consider a  $d$-dimensional stochastic differential equation
\begin{equation} \label{eqn:defX}
X_t = x_0 + \int_0^t b(X_s) ds + \sigma W_t, \quad x_0 \in \br^d, \ t \in [0,T],
\end{equation}
where $\sm$ is a $d\times d$ deterministic, uniformly elliptic matrix, that is for the matrix $a:=\sigma \sigma^*$, there exist $0<\underline{a}<\overline{a}<\infty$ such that for any $\xi \in \br^d$,
\begin{align*}
\underline{a}|\xi|^2 \leq \langle a\xi,\xi \rangle \leq \overline{a}|\xi|^2,
\end{align*}
and $b:\br^d \to \br^d$ is a Borel measurable function. Let $X^h, h >0,$ denote the  Euler-Maruyama approximation of $X$,
\begin{equation} \label{eq:eulerX}
X^h_t = x_0 + \int_0^t b(X^h_{\eta_h(s)})ds +\sigma  W_t, \quad t \in [0,T],
\end{equation}
where $\eta_h(s) = kh$ if $kh \leq s < (k+1)h$ for some nonnegative integer $k$.
In this paper, we  study the convergent rates of the error
$$Err(h) = \bE[f(X)] - \bE [f(X^h)]$$
as $h \to 0$ for some payoff function $f: C[0,T] \to \mathbb{R}$. 

\begin{remark}
Note that the uniformly elliptic condition plays an important role in establishing the convergence of the Euler-Maruyama approximation for SDEs with non-Lipschitz coefficients. In fact, Hairer, Hutzenthaler and Jentzen \cite[Theorem 5.1]{HHJ} constructed a class of $4$-dimensional SDEs whose drift coefficient is a smooth, bounded and non-Lipschitz function, and diffusion coefficient is a deterministic non-uniformly elliptic matrix  for which the Euler-Maruyama scheme does not converges with any polynomial rate, that is
	\begin{align*}
	\lim_{h \searrow 0} \frac{\bE[|X_T-X_T^h|]}{h^{\alpha}}
	=\lim_{h \searrow 0} \frac{|\bE[X_T]-\bE[X_T^h]|}{h^{\alpha}}
	=\left\{ \begin{array}{ll}
	\displaystyle 0 &\text{ if } \alpha=0,  \\
	\displaystyle \infty &\text{ if } \alpha>0.
	\end{array}\right.
	\end{align*}
\end{remark}

A Borel measurable function $\zeta: \br^d \to \br^d$ is called \emph{sub-linear growth} if $\zeta$ is bounded on compact sets and $\zeta (y) = o(|y|)$ as $y \to \infty$.
$\zeta$ is called  \emph{linear growth} if $|\zeta(y)| < c_1|y| + c_2$ for some positive constants $c_1, c_2$. 

\begin{remark}

\begin{itemize}
	\item[(i)]
	It has been shown recently in \cite{HJK} that when $b$ is of super-linear growth, i.e.,  there exist constants $C>0$ and $\theta>1$ such that $|\zeta(y)| \geq |y|^\theta $ for all $|y| > C$, then the Euler-Maruyama approximation (\ref{eq:eulerX}) converges neither in the strong mean square sense nor in  weak sense to the exact solution at a finite time point.
	It means that if $\bE[|X_T|^p]<\infty$ for some $p\in[1, \infty)$ then 
	$$\lim_{h \to 0}\bE\big[|X_T-X^h_T|^p\big] = \infty \quad \text{and} \quad \lim_{h \to 0}\big| \bE\big[|X_T|^p-|X^h_T|^p\big]\big| = \infty.$$
	Thus, in this paper we will consider the case that $b$ is of at most linear growth.
	
	\item[(ii)]
	In the one-dimensional case, $d=1$, it is well-known that if $\sigma \not = 0$ and $b$ is of linear growth,  then the strong existence and path-wise uniqueness hold for the equation (\ref{eqn:defX}) (see \cite{CE}). 
	
	\item[(iii)]
	In the multidimensional case, $d > 1$, it has been shown in \cite{V} that if $b$ is bounded then the equation (\ref{eqn:defX}) has a strong solution and the solution of (\ref{eqn:defX}) is strongly unique. Moreover, if $\sm$ is the identity matrix, then the equation (\ref{eqn:defX}) has a unique strong solution in the class of continuous processes such that 
	$\pr\big( \int_0^T |b(X_s)|^2ds < \infty\big) = 1$
	provided that $\int_{\br^d} |b(y)|^pdy < \infty$ for some $p> d \vee 2 $ (see \cite{KR}).
\end{itemize}

\end{remark}

Throughout this paper, we suppose that equation (\ref{eqn:defX}) has a  weak solution which is unique in the sense of probability law (see Chapter 5 \cite{KS}).


The following results requires no assumption on the smoothness of $f$. 

\begin{theorem} \label{main0}
Let $X$ be a  solution of the SDE \eqref{eqn:defX}.
Suppose that  $b\in \mathcal{B}(\alpha)$ and $b$ is of linear growth. Moreover, assume that $f:C[0,T] \to \br$ is bounded. Then
$$\lim_{h \to 0}\bE[f(X^h)] = \bE[f(X)].$$
\end{theorem}

If $b$ is of sub-linear growth, we can obtain the rate of weak convergence as follows.

\begin{theorem} \label{main1}
Let $X$ be a solution of the SDE \eqref{eqn:defX} and $h>0$.
Suppose that  $b\in \mathcal{B}(\alpha)$ and $b$ is of sub-linear growth.
Moreover, assume that  $f:C[0,T] \to \br$ satisfies  $\bE[| f(x_0+\sm W)|^r]<\infty$ for some $r>2$.
Then there exists a constant $C = C(x_0,b,\sigma,T)$, which depends neither on $h$ nor on $f$, such that 
$$| \bE[f(X)] - \bE [f(X^h)] | \leq C \Big(\bE[| f(x_0+\sm W)|^r] \Big)^{2/r} h^{\frac{\alpha}{2} \wedge \frac 14}.$$
\end{theorem}

The following result concerns with the approximation of maximum  of SDEs.
\begin{corollary}\label{main1.1}
	Assume the hypotheses of Theorem \ref{main1}.
	Moreover, suppose that $g:\br \to \br$ is $\beta$-H\"older continuous with $\beta \in (0,1]$.
	Then there exists  a constant $C$ which does not depend of $h$ such that 
	\begin{align*}
	\left| \bE\left[ g\left( \max_{0\leq s \leq T}|X_s| \right) \right] - \bE\left[g\left(\max_{0\leq s \leq T}|X_{\eta_h(s)}^{h}| \right)\right] \right|
	\leq C \left\{ h^{\frac{\alpha}{2} \wedge \frac{1}{4}} + (	h \log(1/h))^{\beta/2} \right\}.
	\end{align*}
\end{corollary}

For an integral type functional, we obtain the following corollary.

\begin{corollary}\label{main1.5}
Let $h=T/n$ for some $n \in \bN$.
If the drift coefficient $b \in \mathcal{B}(\alpha)$ is bounded, then for any Lipschitz continuous function $f$ and $g \in \mathcal{B}(\beta)$ with $\beta \in (0,1]$, there exists  a constant $C$ which does not depend of $h$ such that 
\begin{align*}
\left| \bE\left[ f\left( \int_0^Tg(X_s)ds \right) \right] - \bE\left[f\left(\int_0^Tg(X_{\eta_h(s)}^h)ds \right)\right] \right|
\leq C h^{\frac{\alpha}{2} \wedge \frac{\beta}{2} \wedge \frac{1}{4}}.
\end{align*}
\end{corollary}

\begin{remark} \label{rk:Macke}
In the paper \cite{M}, the author considered the weak rate of convergence of the Euler-Maruyama scheme for equation (\ref{eqn:defX}) in the case of a one-dimensional diffusion. It was claimed that if $b$ was Lipschitz continuous, the weak rate of approximation is of order $1$. However, we would like to point out that the given proof  contains several gaps (see for instance Lemma 2 of \cite{M} and Remark \ref{rmk:exponent} below) which leave us unsure about the claim. 

\end{remark}

\begin{remark} \label{rk:compareKP}
It has been shown in \cite{KP, MP} that for a stochastic differential equation with $\alpha$-H\"older continuous drift and diffusion coefficients with $\alpha \in (0,1)$, one has
$$| \bE[f(X_T)] - \bE [f(X^h_T)]| \leq C h^{\alpha/2},$$
where $f \in C_b^2$ and the second derivative of $f$ is $\alpha$-H\"older continuous.
On the other hand, in \cite{GR}, Gy\"ongy and R\'asonyi have obtained the strong rate of convergence for  a one-dimensional stochastic differential equation whose drift is  the sum of a Lipschitz continuous and a monotone decreasing H\"older continuous function, and its diffusion coefficient is H\"older continuous.
In \cite{NT}, the authors improve  the results in \cite{GR}.
More precisely, we assume that the drift coefficient $b$ is a bounded and one-sided Lipschitz function, i.e., there exists a positive constant $L$ such that for any $x, y \in \br^d$, $\langle x-y, b(x) - b(y) \rangle_{\br^d} \leq L |x-y|^2$, $b_j \in \mathcal{A}$ for any $j= 1, \ldots, d$ and the diffusion coefficient $\sigma$ is bounded, uniformly elliptic and $(1/2+\alpha)$-H\"older continuous with $\alpha \in [0,1/2]$.
Then for $h=T/n$, it holds that
\begin{align*}
\bE[|X_T - X^h_T|]
&\leq \left\{ \begin{array}{ll}
\displaystyle C(\log 1/h)^{-1} &\text{if } \alpha = 0 \text{ and } d =1,  \\
\displaystyle Ch^{\alpha} &\text{if } \alpha \in (0, 1/2] \text{ and } d=1,\\
\displaystyle Ch^{1/2} &\text{if } \alpha = 1/2 \text{ and } d\geq 2.
\end{array}\right.
\end{align*}
Therefore, if the payoff function $f$ is Lipschitz continuous, {it is straightforward to verify that 
\begin{align*}
|\bE[f(X_T) - f(X^h_T)]|
&\leq \left\{ \begin{array}{ll}
\displaystyle C(\log 1/h)^{-1} &\text{if } \alpha = 0 \text{ and } d =1,  \\
\displaystyle Ch^{\alpha} &\text{if } \alpha \in (0, 1/2] \text{ and } d=1,\\
\displaystyle Ch^{1/2} &\text{if } \alpha = 1/2 \text{ and } d\geq 2.
\end{array}\right.
\end{align*}
}
\end{remark}

In the following  we consider a special case of the functional $f$.  More precisely, we are interested in the law at time $T$ of the diffusion $X$  killed when it leaves an open set. Let $D$ be an open subset of $\mathbb{R}^d$ and denote 
$\tau_D = \inf \{ t >0: X_t \not \in D\}$. Quantities of the  type $\bE[g(X_T)\1_{(\tau_D>T)}]$ appear in many domains, e.g. in financial mathematics when one computes the price of a barrier option on a $d$-dimensional asset price random variable $X_t$ with characteristics $f, T$ and $D$ (see \cite{G00, GoMe} and the references therein for more detail). We approximate $\tau_D$ by 
$\tau^h_D = \inf \{ kh >0: X^h_{kh} \not \in D, k = 0,1,\ldots\}.$

\begin{corollary} \label{main2}
Assume the hypotheses of Theorem \ref{main1}.
Furthermore, we assume 
\begin{enumerate}
\item[(i)] $D$ is of class $C^\infty$ and  $\partial D$ is a compact set (see \cite{GT02} and \cite{G00});
\item[(ii)] $g: \mathbb{R}^d \to \mathbb{R}$ is a measurable function, satisfying $d(\emph{Supp}(g), \partial D) \geq 2\epsilon$ for some $\epsilon >0$ and $\|g\|_\infty = \sup_{x \in \br^d}|g(x)| < \infty$.
\end{enumerate}
Then for any $p>1$, there exist constants $C$ and $C_p$ independent of $h$  such that
\begin{equation} \label{emain2}
 \Big| \bE[g(X_T)\1_{(\tau_D > T)}] - \bE [g(X^h_T)\1_{(\tau^h_D > T)}] \Big| \leq Ch^{\frac{\alpha}{2} \wedge \frac14} + \frac{C_p\|g^p\|_\infty}{1 \wedge \epsilon^{4/p}}h^{\frac{1}{2p}}.
\end{equation}
\end{corollary}

Finally, we consider the approximation for the density of SDE \eqref{eqn:defX}.
Let $p_t(x_0,\cdot)$ and $p_t^{h}(x_0,\cdot)$ be the density functions of $X_t$ and $X_t^{h}$ respectively.
Then we have the following rate of convergence.
\begin{theorem}\label{main2.5}
	Suppose that  $b\in \mathcal{B}(\alpha)$ and bounded.
	Then for any $p>d$ and $r>1$, there exists constants $C_{p,r}$ and $c_p$ such that for any $y \in \br^d$ and $h\in (0,T/2)$,
	\begin{align*}
	|p_T(x_0,y)-p_T^h(x_0,y)|
	\leq  C_{p,r}g_{c_pT}(x_0,y) \left\{\frac{h}{T^{1/2}}+h^{\alpha/2}+h^{1/(2pr)} \right\},
	\end{align*}
	where $g_{c_pT}(x_0,y)$ is the density function of normal distribution defined in \eqref{nomal}.
\end{theorem}

\begin{remark}
	(i) Note that if $d=1$, we can choose $p \in (1,2)$ and $r=2/p$, and then
	\begin{align*}
	|p_T(x_0,y)-p_T^h(x_0,y)|
	&\leq  C_{p,r}g_{c_pT}(x_0,y) \left\{\frac{h}{T^{1/2}}+h^{\alpha/2}+h^{1/4} \right\}.
	\end{align*}
	(ii) Konakov and Menozzi \cite{KoMe} obtain a better rate of convergence under a further assumption that the drift coefficient is  piecewise smooth (see Theorem 2 in \cite{KoMe}).
	However, in our setting, the drift coefficient may have infinite number of discontinuous points.
\end{remark}

\subsection{Weak approximation of reflected stochastic differential equations} \label{subs:rsde}

We first recall the Skorohod problem.
\begin{lemma}[\cite{KS}, Lemma III.6.14] \label{Skorohod}
Let $z \geq 0$ be a given number and  $y :[0,\infty) \to \br$ be a continuous function with $y_0=0$.
Then there exists a unique continuous function $\ell=(\ell_t)_{t \geq 0}$ satisfying the following conditions:
\begin{itemize}
\item[(i)]   $x_t:=z+ y_t+\ell_t \geq 0, 0 \leq t < \infty$; 
\item[(ii)] $\ell$ is a non-decreasing function with $\ell_0=0$ and  $\ell_t=\int_0^t \1(x_s=0) d\ell_s$.
\end{itemize}
Moreover, $\ell=(\ell_t)_{t \geq 0}$ is given by
\begin{align*}
\ell_t = \max\{ 0, \max_{0\leq s \leq t} (-z-y_s)\}
    = \max_{0\leq s \leq t} \max(0,\ell_s-x_s).
\end{align*}
\end{lemma}

Let us consider the following one-dimensional reflected stochastic differential equation valued in $[0,\infty)$ such that
\begin{align} \label{RSDE} 
\begin{cases} X_t &= x_0 + \int_0^t b(X_s) ds + \sigma W_t + L_t^0(X), x_0 \in [0,\infty), t \in [0,T],\\
L_t^0(X) &= \int_0^t \1_{(X_s = 0)} dL_s^0(X), 
\end{cases}
\end{align}
where $(L_t^0(X))_{0 \leq t \leq T}$ is a non-decreasing continuous process stating at the origin and
$L_t^{0}(X)$ is called the symmetric local time of $X$ up to time $t$ at the level $0$.
Equation \eqref{RSDE} will be studied under the assumptions that $b$ is of sub-linear growth and H\"older continuous and $\sigma$ is a non-zero constant.
Moreover, we assume that the SDE (\ref{RSDE}) has a weak solution and the uniqueness in the sense of probability law holds (see \cite{RS, T}).
Using Lemma \ref{Skorohod}, we have
\begin{align*}
L_t^0(X) = \sup_{0 \leq s \leq t} \max \left(0, L_s^0(X)- X_s  \right).
\end{align*}

Now we define the Euler-Maruyama scheme $X^h=(X^h_t)_{0\leq t \leq T}$ for the reflected stochastic differential equation (\ref{RSDE}).
Let $X_0^h := x_0 $ and define
\begin{align}\label{EM_RSDE}
X_t^h &= x_0 + \int_0^t b(X_{\eta_h(s)}^h) ds + \sigma W_t + L_t^0(X^h).
\end{align}
The existence of the pair $(X_t^h, L_t^0(X^h))_{0\leq t \leq T}$ is deduced from  Lemma \ref{Skorohod}. 
Moreover
$$L_t^0(X^h) = \int_0^t \1_{(X_s^h = 0)} dL_s^0(X^h).$$
By the definition of the Euler-Maruyama scheme, we have the following representation.
For each $k=0, 1,\ldots$,
\begin{align*}
X_{(k+1)h}^h = X_{kh}^h + b(X_{kh}^h) h + \sigma (W_{(k+1)h}-W_{kh}) + \max(0, A_k-X_{kh}^h),
\end{align*}
where
\begin{align*}
A_k:= \sup_{kh \leq s < (k+1)h} \left(-b(X_{kh}^h)(s-kh) - \sigma(W_s-W_{kh}) \right).
\end{align*}
Though $A_k$ is defined by the  supremum of a stochastic process, it can be simulated by using the following lemma.
\begin{lemma}[\cite{Le93}, Theorem 1]
Let $t \in [0,T]$ and $a,c \in \br$.
Define $S_t:= \sup_{0 \leq s \leq t} (a W_s + cs )$.
Let $U_t$ be a centered Gaussian random variable with variance $t$ and let $V_t$ be an exponential random variable with parameter $1/(2t)$ independent from $U_t$.
Define
\begin{align*}
Y_t:=\frac12 (aU_t + ct + (a^2 V_t + (aU_t+ct)^{2})^{1/2}).
\end{align*}
Then the vectors $(W_t, S_t)$ and $(U_t,Y_t)$ have the same law.
\end{lemma}

Under the Lipschitz condition for the coefficients of the reflected SDE (\ref{RSDE}), L\'epingle \cite{Le95} shows that
\begin{align*}
\bE[\sup_{0\leq t \leq T}|X_t-X_t^h|^2]^{1/2} \leq Ch^{1/2},
\end{align*}
for some constant $C$.

We obtain the following result on the weak convergence for the Euler-Maruyama scheme for a reflected SDE with non-Lipschitz coefficient.
\begin{theorem} \label{main3}
Let $X$ be a solution of the SDE \eqref{RSDE} and $h>0$.
Suppose that the drift coefficient $b$ is of sub-linear growth  and $\alpha$-H\"older continuous with $\alpha \in (0,1]$.
Moreover, assume that $f:C[0,T] \to \br$ is bounded. 
Then there exists a constant $C$, which does not depend on $h$, such that 
$$| \bE[f(X)] - \bE [f(X^h)] | \leq C h^{\alpha/2}.$$
\end{theorem}

\section{Proofs} \label{sec:proofs}
From now on, we will repeatedly use without mentioning the following elementary estimate
\begin{equation} \label{exp_inq}
  \sup_{x \in \br} |x|^p e^{k|x| - x^2}  < \infty, \quad \text{for any } p\geq 0, \ k \in \br .
 \end{equation}
 Throughout this section, a symbol $C$ stands for a positive generic constant independent of the discretization parameter $h$, which nonetheless may depend on time $T$, coefficients $b, \sm$ and payoff function $f$.

\subsection{Change of Measures} \label{subs:change}
Let $X$ be a solution of the SDE \eqref{eqn:defX} defined in a certain probability space $(\Omega, \mf, \pr)$ with a Brownian motion $(W_t, \mf_t)_{t \geq 0}$. The notations $x_0, b, \sigma, \eta_h$ are defined as at the beginning of  Section \ref{subs:sde}.
From now on, we will use the following notations
\begin{align*}
& Z_t = e^{Y_t}, \quad Y_t = \int_0^t (\sm^{-1}b)_j(x_0 + \sm W_s)dW^j_s - \frac{1}{2}\int_0^t |\sm^{-1}b(x_0 + \sm W_s)|^2 ds,\\
& Z^h_t = e^{Y^h_t}, \quad Y^h_t = \int_0^t (\sm^{-1}b)_j(x_0 + \sm W_{\eta_h(s)})dW^j_s - \frac{1}{2}\int_0^t |\sm^{-1} b(x_0 + \sm W_{\eta_h(s)})|^2 ds,
\end{align*}
where we use Einstein's summation convention on repeated indices. We also use the following auxiliary stopping times
$$\tau^W_D = \inf\{t \geq 0: x_0 + \sigma W_t \not \in D\}, \ \text{and} \ \tau^{W,h}_D = \inf\{kh \geq 0: x_0 + \sigma W_{kh} \not \in D, k = 0,1,\ldots\}.$$

\begin{lemma} \label{lem:decom}
Let $X$ be a solution of the SDE \eqref{eqn:defX} and $h>0$.
Suppose that $b$ is a function with at most linear growth, then we have  the following representations
\begin{align}
& \bE[f(X)] - \bE [f(X^h)] = \bE[ f(x_0 + \sm W )(Z_T - Z^h_T)],  \label{eq:error}
\end{align}
and 
\begin{align}
  \bE[g(X_T)&\1_{(\tau_D > T)}] - \bE [g(X^h_T)\1_{(\tau^h_D > T)}] \notag \\
\qquad & = \bE[ g(x_0 + \sm W_T )(Z_T \1_{(\tau^W_D > T)} - Z^h_T\1_{(\tau^{W,h}_D > T)})],  \label{eq:error2}
\end{align}
for all measurable functions $f: C[0,T] \to \br$ and $g: \br^d \to \br$ provided that all the above expectations are integrable.
\end{lemma}

\begin{proof}
Let $\sm^{-1}$ be the inverse matrix of $\sm$. 
Since $b$ is of linear growth, so is $\sm^{-1}b$. Thus, there exist constants $c_1, c_2 > 0$ such that $|b(x)| < c_1|x|+ c_2$ for any $x \in \br^d$. For any $0 \leq  t \leq  t_0 \leq T $, 
\begin{align*}
|X_t| &\leq |x_0| + |\sm W_t| + \int_0^t |b(X_s)| ds\\
&\leq |x_0| + |\sm|\sup_{0\leq s \leq t_0}|W_s| + c_2 t_0 + c_1\int_0^t |X_s| ds.
\end{align*}
Applying Gronwall's inequality for $t \in [0, t_0]$, one obtains
\begin{align} \label{eq:Xbound}
|X_{t_0}| &\leq  (|x_0| + |\sm|\sup_{0\leq s \leq t_0}|W_s| + c_2t_0)e^{c_1 t_0} \notag \\ 
&\leq  (|x_0| + c_2T)e^{c_1 T} + |\sm|e^{c_1T} \sup_{0\leq s \leq t_0}|W_s|. 
\end{align}
On the other hand, for each  integer $k \geq 1$, one has
\begin{align*}
|X^h_{kh}| &\leq |X^h_{(k-1)h}| + h |b(X^h_{(k-1)h})| + 2|\sm|\sup_{0 \leq t \leq kh} |W_t|\\
&\leq (1+ hc_1) |X^h_{(k-1)h}| + hc_2 + 2|\sm|\sup_{0 \leq t \leq kh} |W_t|.
\end{align*}
Hence, a simple iteration yields that 
$$|X^h_{kh}| \leq (1+hc_1)^k|x_0| + (hc_2 + 2|\sm|\sup_{0 \leq t \leq kh} |W_t|)\frac{(1+hc_1)^{k -1}-1}{hc_1}.$$ 
Thus, for any $t \in (0,T]$,
$$|X^h_{\eta_h(t)}| \leq (1+ hc_1)^{T/h}|x_0| + \frac{c_2(1+hc_1)^{T/h}}{c_1}+ 2|\sm| \frac{(1+hc_1)^{T/h}}{hc_1}\sup_{0 \leq s \leq \eta_h(t)}|W_s|.$$ 
Moreover, 
$$|X^h_t - X^h_{\eta_h(t)}| \leq c_1 h |X^h_{\eta_{h}(t)}| + c_2h + 2|\sigma| \sup_{0\leq s\leq t} |W_t|.$$
Therefore, for any $t \in (0,T]$, we have
\begin{equation} \label{eq:Xhbound}
|X^h_t| \leq  (1+c_1h)^{1+T/h}\frac{c_1|x_0|+ c_2}{c_1} + c_2h + \frac{2|\sm|(1+hc_1)^{1+T/h}+ 2hc_1}{hc_1} \sup_{0 \leq s \leq t}|W_s|.
\end{equation}
We define new measures $\bQ$ and $\bQ^h$ as 
\begin{align*}
&\frac{d \bQ}{d \pr} =  \exp \Big( -\int_0^T (\sm^{-1} b)_j(X_s) dW^j_s - \frac{1}{2} \int_0^T |\sm^{-1}b(X_s)|^2 ds\Big),\\
&\frac{d \bQ^h}{d \pr} =  \exp \Big( -\int_0^T (\sm^{-1} b)_j(X_{\eta_h(s)}^h) dW^j_s - \frac{1}{2} \int_0^T |\sm^{-1}b(X_{\eta_h(s)}^h)|^2 ds\Big).
\end{align*}
It follows from Corollary 3.5.16 in \cite{KS} together with  estimates (\ref{eq:Xbound}) and (\ref{eq:Xhbound}) that $\bQ$ and $\bQ^h$ are probability measures. Furthermore, it follows from Girsanov theorem that processes $B = \{(B^1_t, \ldots, B^d_t), \ 0 \leq t \leq T\}$ and $B^h = \{(B^{h,1}_t, \ldots, B^{h,d}_t), \ 0 \leq t \leq T\}$ 
defined by
$$B^j_t = W^j_t + \int_0^t (\sm^{-1}b)_j(X_s)ds, \ B^{h,j}_t = W^j_t + \int_0^t (\sm^{-1}b)_j(X_{\eta_h(s)})ds, \ 1 \leq j \leq d, 0\leq t \leq T,$$ 
are $d$-dimensional Brownian motions with respect to $\bQ$ and $\bQ^h$, respectively.  Note that $X_s = x_0 + \sigma B_s$ and $X^h_s = x_0 + \sigma B^h_s$. Therefore, 
\begin{align*}
\bE[f(X)] &= \bE_\bQ \Big[f(X)\frac{d\pr}{d\bQ}\Big] \\
&= \bE_\bQ \Big[ f(x_0 + \sigma B) \exp \Big( \int_0^T (\sigma^{-1}b)_j(x_0+\sigma B_s)dB^j_s - \frac12 \int_0^T |\sigma^{-1}b(x_0+\sigma B_s)|^2ds\Big)\Big]\\
&= \bE [f(x_0+\sigma W)Z_T].
\end{align*}
Repeating the previous argument leads to $\bE[f(X^h)] = \bE [f(x_0+\sigma W)Z^h_T],$ which implies  \eqref{eq:error}. The proof of   \eqref{eq:error2} is similar and will be omitted. 
\end{proof}

From now on, we will use the representation formulas in Lemma \ref{lem:decom} to analyze the weak rate of convergence. We need the following estimates on the moments of $Z$ and $Z^h$.
\begin{lemma} \label{lem:powZ}
Suppose that $b$ is of sub-linear growth. Then for any $p>0$, 
$$\bE[|Z_T|^p + |Z^h_T|^p] \leq C < \infty,$$
for some constant $C$ which is not depend on $h$.
\end{lemma}
\begin{proof}
It suffices to proof the statement for $p \geq 1$.
Using H\"older's inequality, we have
\begin{align*}
\bE [e^{pY_T}] &= \bE\Big[\exp\Big( p\int_0^T (\sm^{-1}b)_j(x_0+\sm W_s) dW^j_s - \frac{p}{2} \int_0^T |\sm^{-1}b(x+\sm W_s)|^2ds \Big) \Big] \\
&= \bE\Big[ \exp \Big( p\int_0^T (\sm^{-1}b)_j(x_0+\sm W_s) dW^j_s -  p^2\int_{t_{n-1}}^{t_n} |\sm^{-1}b(x_0+\sm W_s)|^2ds +\\
&\qquad + ( p^2 - \frac{p}{2})\int_0^T |\sm^{-1}b(x_0+\sm W_s)|^2ds \Big)\Big]\\
&\leq  \Big\{ \bE \Big[  \exp \Big( 2 p\int_0^T (\sm^{-1}b)_j(x_0+\sm W_s) dW^j_s - 2 p^2\int_0^T |\sm^{-1} b(x_0+\sm W_s)|^2ds \Big) \Big] \Big\}^{1/2} \\
&\qquad  \times \Big\{ \bE \Big[ \exp \Big( (2p^2-p) \int_0^T |\sm^{-1} b(x_0+\sm W_s)|^2 ds \Big) \Big]\Big\}^{1/2}.
\end{align*}
Since $b$ is of linear growth, so is $\sm^{-1}b$ and it follows from Corollary 3.5.16 in \cite{KS} that 
\begin{equation} \label{eq:nov}
\bE \Big[  \exp \Big( 2 p\int_0^T (\sm^{-1} b)_j(x_0+\sm W_s) dW^j_s - 2 p^2\int_0^T  |\sm^{-1}b(x_0+\sm W_s)|^2ds \Big) \Big]  = 1.
\end{equation}
On the other hand, since $b$ is bounded on compact sets and $b(y) = o(|y|)$ as $y \to \infty$, for any $\delta>0$ sufficiently small, there exists a constant $M>0$ such that $|\sm^{-1} b(x_0+\sm y)|^2 \leq \delta |y|^2 + M$ for any $y \in \br^d$. Thus,
\begin{align*}
\int_0^T |\sm^{-1}b(x_0+\sm W_s)|^2 ds &\leq \int_0^T (\delta |W_s|^2 + M) ds \leq TM + T\delta \sup_{s \leq T}|W_s|^2 \\
& \leq TM + T\delta\sum_{j=1}^d ( (\sup_{s\leq T} W^j_s)^2 + (\inf_{s \leq T} W^j_s)^2).
\end{align*}
Hence, 
\begin{align*}
&\bE \Big[ \exp \Big( (2p^2-p) \int_0^T |\sm^{-1} b(x_0+\sm W_s)|^2 ds \Big) \Big] \\
&\leq e^{(2p^2-p)MT} \bE\Big[ \exp \Big( T\delta(2p^2-p) \sum_{j=1}^d( (\sup_{s\leq T} W^j_s)^2 + (\inf_{s \leq T} W^j_s)^2)\Big) \Big]\\
&\leq  e^{(2p^2-p)MT} \Big(\bE\Big[ \exp \Big( 2T\delta(2p^2-p) |W^1_T|^2\Big) \Big]\Big)^{d/2},
\end{align*}
where the last inequality follows from H\"older's inequality and  the fact that $$\sup_{s\leq T}W^j_s \overset{law}{=}-\inf_{s\leq T}W^j_s \overset{law}{=} |W^1_T|.$$ 
Because $\bE\Big[ \exp \Big( 2T\delta(2p^2-p)W_T^2\Big) \Big] < \infty$ if one chooses $\delta < (4T^2(2p^2-p))^{-1}$, we obtain $\bE[|Z_T|^p] < \infty$.
Furthermore, since equation (\ref{eq:nov}) still holds if one replaces $b(x_0+\sm W_s)$ with $b(x_0+\sm W_{\eta_h(s)})$, a similar argument yields $\bE[|Z^h_T|^p] < \infty$.
\end{proof}

\begin{remark} \label{rmk:exponent}
In general, the conclusion of Lemma \ref{lem:powZ} is no longer correct if we only suppose that $b$ is of linear growth or even Lipschitz. 

Indeed, consider the particular case that $d=1,\sigma = 1$ and  $b(x) = x $, which is a Lipschitz function. It follows from H\"older's inequality that
\begin{align*}
&\bE \Big[ \exp\Big( \frac{p}{2}\int_0^T W_s^2 ds\Big) \Big] \bE\Big[\exp\Big( p\int_0^T W_s dW_s - \frac{p}{2} \int_0^T W_s^2ds \Big) \Big]\\
 &= e^{-pT/2}\bE \Big[ \exp\Big( \frac{p}{2}\int_0^T W_s^2 ds\Big) \Big]\bE \Big[ \exp \Big( \frac{p}{2}W_T^2 - \frac{p}{2}\int_0^T W_s^2 ds \Big) \Big]\\
&\geq e^{-pT/2}\Big( \bE\big[ e^{pW_T^2/4}\big]\Big)^2.
\end{align*}
 Furthermore, for any $p, T>0$ such that $pT\geq 2$ and $pT^2 < 1/2$, we have $\bE\big[ e^{pW_T^2/4}\big] = \infty$, whereas 
 \begin{align*}
 \bE \Big[ \exp\Big( \frac{p}{2}\int_0^T W_s^2 ds\Big) \Big] &\leq \bE \Big[ \exp\Big(\frac{pT}{2}\sup_{s\leq T}|W_s|^2\Big) \Big]\\
 & \leq \bE \Big[ \exp\Big(\frac{pT}{2}(\sup_{s\leq T} W_s)^2 + \frac{pT}{2}(\inf_{s \leq T} W_s)^2)\Big) \Big]\\
& \leq \Big( \bE \big[ e^{pT|W_T|^2} \big]\Big)^2< \infty.
 \end{align*}
 Therefore, 
$$ \bE\Big[\exp\Big( p\int_0^T W_s dW_s - \frac{p}{2} \int_0^T W_s^2ds \Big) \Big]=\infty, \quad \text{ if \ \  } pT\geq 2, \ pT^2 < \frac{1}{2}.$$
\end{remark}

\subsection{Some auxiliary estimates} \label{subs:auxi}
We recall that $\mathcal{A}$ is a class of functions satisfying \eqref{repX} and $W$ is a $d$-dimensional Brownian motion.
The following result plays a crucial role in our argument.
\begin{lemma} \label{lem:auxi}
For any $\zeta \in \mathcal{A}$, any $p\geq 1, \ t > s > 0$,
\begin{equation}\label{ulg1}
\bE [|\zeta(W_t) - \zeta(W_s)|^p] \leq C_p \frac{\sqrt{t-s}}{\sqrt{s}},
\end{equation}
for some constant $C_p$ not depending on neither $t$ nor $s$. On the other hand, if $\zeta$ is $\alpha$-H\"older continuous then 
\begin{equation}\label{ulg2}
\bE [|\zeta(W_t) - \zeta(W_s)|^p]  \leq C_p (t-s)^{p/2}.
\end{equation}

\end{lemma}
\begin{proof}
If $\zeta \in \mathcal{A}$, let 
$(\zeta_N)_{N \in \bN}$ be an approximating sequence of $\zeta$ in $\mathcal{A}$. Since $\zeta_N \to \zeta$ in $L^1_{loc}(\br^d)$ and $\zeta$ and $\zeta_N$ are uniformly exponential bounded, we have
\begin{equation} \label{eq:g1}
\bE [|\zeta(W_t) - \zeta(W_s)|^p)] = \lim_{N \to \infty} \bE [|\zeta_N(W_t) - \zeta_N(W_s)|^p].
\end{equation}
Next, we will show that
\begin{equation} \label{eq:g2}
\sup_{N \in \bN} \bE [|\zeta_N(W_t) - \zeta_N(W_s)|^p]   \leq C \frac{\sqrt{t-s}}{\sqrt{s}}. 
\end{equation}
Indeed,  we write
\begin{align*}
&\bE [|\zeta_N(W_t) - \zeta_N(W_s)|^p] \\
 =& \int_{\br^d}\int_{\br^d} |\zeta_N(x+y) - \zeta_N(x)|^p \frac{e^{-|x|^2/2s}}{(2\pi s)^{d/2}}  \frac{e^{-|y|^2/2(t-s)}}{(2\pi (t-s))^{d/2}}  dxdy\\
\leq& C\int_{\br^d}\int_{\br^d}|\zeta_N(x+y) - \zeta_N(x)| (e^{K(|x+y|}+ e^{K|y|})^{p-1} \frac{e^{-|x|^2/2s}}{(2\pi s)^{d/2}}  \frac{e^{-|y|^2/2(t-s)}}{(2\pi (t-s))^{d/2}}  dxdy\\
\leq & C  \sum_{i=1}^d  \int_{\br^d}\int_{\br^d} \int_0^1  \Big| y_i \frac{\partial \zeta_N(x+\theta y)}{\partial x_i}\Big| e^{K(p-1)(|x|+2|y|)}\frac{e^{-|x|^2/2s}}{(2\pi s)^{d/2}}  \frac{e^{-|y|^2/2(t-s)}}{(2\pi (t-s))^{d/2}}  d\theta dxdy\\ 
\leq & C\sqrt{t-s}   \sum_{i=1}^d \int_{\br^d}\int_{\br^d} \int_0^1 \Big| \frac{\partial \zeta_N(x+\theta y)}{\partial x_i}\Big| \frac{e^{-|x|^2/4s}}{(2\pi s)^{d/2}}  \frac{e^{-|y|^2/4(t-s)}}{(2\pi (t-s))^{d/2}}  d\theta dxdy.
\end{align*}
It follows from $\mathcal{A}(iii)$ that 
\begin{align*}
\sup_N \sum_{i=1}^d \int_{\br^d} \Big| \frac{\partial \zeta_N(x+\theta y)}{\partial x_i}\Big|
\frac{e^{-|x|^2/4 s }}{(2\pi s )^{(d-1)/2}}  dx \leq C e^{K|\theta y |},
\end{align*}
thus
\begin{align*}
\bE [|\zeta_N(W_t) - \zeta_N(W_s)|^p]  &\leq  C\frac{\sqrt{t-s}}{\sqrt{s}}   \int_{\br^d} \int_0^1 Ce^{K|\theta y|}  \frac{e^{-|y|^2/4(t-s)}}{(2\pi (t-s) )^{d/2}}   d\theta dy \\
&\leq C\frac{\sqrt{t-s}}{\sqrt{s}}.
\end{align*}
From (\ref{eq:g1}) and (\ref{eq:g2}) we get \eqref{ulg1}. The proof of \eqref{ulg2} is straightforward.
\end{proof}

\begin{lemma}\label{lem:auxi2}
	Suppose $\zeta_A\in \mathcal{A}$ and $\zeta_H:\br^d \to \br$ is $\alpha$-H\"older continuous with $\alpha \in (0,1]$.
	Let $M$ be a non-negative constant.
	Then there exists $C>0$ such that for any $0<t_1<t_2<t_3<t_4 \leq T$,
	\begin{align}
	\bE \left[|\zeta_A(W_{t_2}) - \zeta_A(W_{t_1})||\zeta_A(W_{t_4}) - \zeta_A(W_{t_3})| \sum_{i=1}^4e^{M|W_{t_i}|}\right]
		&\leq \frac{C\sqrt{t_4-t_3} \sqrt{t_2-t_1}}{\sqrt{t_3-t_2}\sqrt{t_1}}, \label{ulg21}\\
	\bE \left[|\zeta_A(W_{t_2}) - \zeta_A(W_{t_1})||\zeta_H(W_{t_4}) - \zeta_H(W_{t_3})| \sum_{i=1}^4e^{M|W_{t_i}|}\right]
		&\leq \frac{C(t_4-t_3)^{\alpha/2} \sqrt{t_2-t_1}}{\sqrt{t_1}}, \label{ulg22}\\
	\bE \left[|\zeta_H(W_{t_2}) - \zeta_H(W_{t_1})||\zeta_A(W_{t_4}) - \zeta_A(W_{t_3})| \sum_{i=1}^4e^{M|W_{t_i}|}\right]
		&\leq \frac{C\sqrt{t_4-t_3} (t_2-t_1)^{\alpha/2} }{\sqrt{t_3-t_2}}, \label{ulg23}\\
	\bE \left[|\zeta_H(W_{t_2}) - \zeta_H(W_{t_1})||\zeta_H(W_{t_4}) - \zeta_H(W_{t_3})| \sum_{i=1}^4e^{M|W_{t_i}|}\right]
		&\leq C(t_4-t_3)^{\alpha/2}(t_2-t_1)^{\alpha/2} \label{ulg24}.
	\end{align}
\end{lemma}
\begin{proof}
	Let $(\zeta_{A,N})_{N \in \bN}$ be an approximating sequence of $\zeta_A$ in $\mathcal{A}$.
	Since $\zeta_{A,N} \to \zeta_A$ in $L^1_{loc}(\br^d)$ and $\zeta_A$ and $\zeta_{A,N}$ are uniformly exponential bounded, we have
	\begin{align} \label{eq:g12}
	&\bE \left[|\zeta_A(W_{t_2}) - \zeta_A(W_{t_1})||\zeta_A(W_{t_4}) - \zeta_A(W_{t_3})| \sum_{i=1}^4e^{M|W_{t_i}|}\right] \notag\\
	&=\lim_{N \to \infty} \bE \left[|\zeta_{A,N}(W_{t_2}) - \zeta_{A,N}(W_{t_1})||\zeta_{A,N}(W_{t_4}) - \zeta_{A,N}(W_{t_4})| \sum_{i=1}^4e^{M|W_{t_i}|}\right].
	\end{align}
	Next, we will show that
	\begin{equation}\label{eq:g22}
	\sup_{N \in \bN} \bE\left[|\zeta_{A,N}(W_{t_2}) - \zeta_{A,N}(W_{t_1})||\zeta_{A,N}(W_{t_4}) - \zeta_{A,N}(W_{t_4})| \sum_{i=1}^4e^{M|W_{t_i}|}\right]
	\leq \frac{C\sqrt{t_4-t_3} \sqrt{t_2-t_1}}{\sqrt{t_3-t_2}\sqrt{t_1}}. 
	\end{equation}
	We observe that
	\begin{align}\label{eq:g3}
	&\bE\left[|\zeta_{A,N}(W_{t_2}) - \zeta_{A,N}(W_{t_1})||\zeta_{A,N}(W_{t_4}) - \zeta_{A,N}(W_{t_4})| \sum_{i=1}^4e^{M|W_{t_i}|}\right] \notag\\
	&=\int_{\br^d}dx \int_{\br^d}dy \int_{\br^d}dz \int_{\br^d}dw
	|\zeta_{A,N}(z+w) - \zeta_{A,N}(w)||\zeta_{A,N}(x+y+z+w) - \zeta_{A,N}(y+z+w)| \notag\\
	&\quad \times \{e^{M|w|}+e^{M|z+w|}+e^{M|y+z+w|}+e^{M|x+y+z+w|} \}
	g_{t_4-t_3}(x) g_{t_3-t_2}(y) g_{t_2-t_1}(z) g_{t_1}(w) \notag\\
	&\leq C\int_{\br^d}dx \int_{\br^d}dy \int_{\br^d}dz \int_{\br^d}dw
	|\zeta_{A,N}(z+w) - \zeta_{A,N}(w)||\zeta_{A,N}(x+y+z+w) - \zeta_{A,N}(y+z+w)| \notag\\
	&\quad \times 
	g_{c(t_4-t_3)}(x) g_{c(t_3-t_2)}(y) g_{c(t_2-t_1)}(z) g_{ct_1}(w).
	\end{align}
	Using the mean-value theorem, \eqref{eq:g3} is bounded by
	\begin{align}\label{eq:g4}
	&\int_{\br^d}dx \int_{\br^d}dy \int_{\br^d}dz \int_{\br^d}dw \int_{0}^{1} d\theta \int_{0}^{1} d \delta
	\left|z_i \frac{\partial \zeta_{A,N}(w+\theta z)}{\partial w_i}\right|
	\left|x_i \frac{\partial \zeta_{A,N}(y+z+w+\delta x)}{\partial y_i}\right| \notag\\
	&\quad \times g_{c(t_4-t_3)}(x) g_{c(t_3-t_2)}(y) g_{c(t_2-t_1)}(z) g_{ct_1}(w) \notag\\
	&\leq C\sqrt{t_4-t_3} \sqrt{t_2-t_1} \int_{\br^d}dx \int_{\br^d}dy \int_{\br^d}dz \int_{\br^d}dw \int_{0}^{1} d\theta \int_{0}^{1} d \delta \notag\\
	&\quad \times \left|\frac{\partial \zeta_{A,N}(w+\theta z)}{\partial w_i}\right|
	\left|\frac{\partial \zeta_{A,N}(y+z+w+\delta x)}{\partial y_i}\right|
	g_{c(t_4-t_3)}(x) g_{c(t_3-t_2)}(y) g_{c(t_2-t_1)}(z) g_{ct_1}(w)
	\end{align}
	It follows from $\mathcal{A}(iii)$ that 
	\begin{align*}
	\int_{\br^d} \left|\frac{\partial \zeta_{A,N}(w+\theta z)}{\partial w_i}\right|
	\frac{e^{-\frac{|w|^2}{2ct_1}}}{(2 c\pi t_1)^{(d-1)/2}} dw
	&\leq C e^{K|\theta z|},\\
	\int_{\br^d} \left|\frac{\partial \zeta_{A,N}(y+z+w+\delta x)}{\partial y_i}\right|
	\frac{e^{-\frac{|y|^2}{2c(t_3-t_2)}}}{(2 c\pi (t_3-t_2))^{(d-1)/2}} dy
	&\leq C e^{K|z+w+\delta x|},
	\end{align*}
	thus \eqref{eq:g4} is bounded by
	\begin{align*}
	&\frac{C\sqrt{t_4-t_3} \sqrt{t_2-t_1}}{\sqrt{t_3-t_2}} \int_{\br^d}dx \int_{\br^d}dz \int_{\br^d}dw \int_{0}^{1} d\theta \left|\frac{\partial \zeta_{A,N}(w+\theta z)}{\partial w_i}\right|
	g_{c(t_4-t_3)}(x) g_{c(t_2-t_1)}(z) g_{ct_1}(w) \notag\\
	& \leq \frac{C\sqrt{t_4-t_3} \sqrt{t_2-t_1}}{\sqrt{t_3-t_2} \sqrt{t_1}} \int_{\br^d}dx \int_{\br^d}dz
	g_{c(t_4-t_3)}(x) g_{c(t_2-t_1)}(z)\\
	&=\frac{C\sqrt{t_4-t_3} \sqrt{t_2-t_1}}{\sqrt{t_3-t_2} \sqrt{t_1}}.	
	\end{align*}
	From \eqref{eq:g12} and \eqref{eq:g22} we get \eqref{ulg21}.
	
	The proofs of \eqref{ulg22}, \eqref{ulg23} and \eqref{ulg24} follow from similar arguments.
\end{proof}

In order to prove some key estimations in Lemma \ref{dai5}, we need the following inequality. Recall that $\eta_h(s)$ is defined in \eqref{def:EulerX}.

\begin{lemma}\label{APP_1}
	Let $n $ be a natural number such that $(n-1)h< T \leq nh$.
	Define $t_i^{h}=ih$ for $i=0,\ldots,n-1$ and $t_n^{h}=T$.
	Then it holds that 
	\begin{align*}
	\sum_{i=1}^{n-1} \int_{t_i^{h}}^{t_{i+1}^{h}}du \int_{t_1^{h}}^{t_i^{h}}ds
	\left\{
	\frac{1}{\sqrt{\eta_h(u)-s}\sqrt{\eta_h(s)}}
	+\frac{1}{\sqrt{\eta_h(s)}}
	+\frac{1}{\sqrt{\eta_h(u)-s}}
	\right\}
	\leq 4\sqrt{2}+2\sqrt{T}+\frac{4T^{3/2}}{3}.
	\end{align*}
	\begin{proof}
		We first note that if $s\geq t_1^h$ then $\eta_h(s) \geq s/2$.
		The first integral is estimated as follows	
		\begin{align*}
		\sum_{i=1}^{n-1} \int_{t_i^{h}}^{t_{i+1}^{h}}du \int_{t_1^{h}}^{t_i^{h}}ds
		\frac{1}{\sqrt{\eta_h(u)-s}\sqrt{\eta_h(s)}}
		&\leq \sqrt{2} h \sum_{i=2}^{n-1} \left\{ \int_{t_1^{h}}^{t_i^{h}/2}ds+\int_{t_i^{h}/2}^{t_i^{h}}ds\right\}
		\frac{1}{\sqrt{t_i^h-s}\sqrt{s}}\\
		&\leq \sqrt{2} h \sum_{i=2}^{n-1} \sqrt{\frac{2}{t_i}} \left\{ \int_{t_1^{h}}^{t_i^{h}/2} \frac{1}{\sqrt{s}}ds
		+\int_{t_i^{h}/2}^{t_i^{h}} \frac{1}{\sqrt{t_i^h-s}} ds\right\}\\
		&\leq 4\sqrt{2}.
		\end{align*}
		The second integral is estimated as follows
		\begin{align*}
		\sum_{i=1}^{n-1} \int_{t_i^{h}}^{t_{i+1}^{h}}du \int_{t_1^{h}}^{t_i^{h}}ds
		\frac{1}{\sqrt{\eta_h(s)}}
		&\leq \sqrt{2} h \sum_{i=1}^{n-1} \int_{t_1^{h}}^{t_i^{h}} \frac{ds}{\sqrt{s}}
		\leq 2\sqrt{T}.
		\end{align*}
		The third integral is estimated as follows
		\begin{align*}
		\sum_{i=1}^{n-1} \int_{t_i^{h}}^{t_{i+1}^{h}}du \int_{t_1^{h}}^{t_i^{h}}ds
		\frac{1}{\sqrt{\eta_h(u)-s}}
		&\leq h \sum_{i=2}^{n-1} \int_{t_1^{h}}^{t_i^{h}}	\frac{ds}{\sqrt{t_i^h-s}}
		=2h \sum_{i=2}^{n-1} \sqrt{t_i^h-h}
		=2 \sum_{i=2}^{n-1} \int_{t_{i-1}^h}^{t_{i}^h} \sqrt{t_{i-1}^h} ds\\
		&\leq 2 \int_{0}^{T} \sqrt{s}ds
		=\frac{4T^{3/2}}{3}.
		\end{align*}
		This concludes the proof of the statement.
	\end{proof}
	
\end{lemma}

Now by using Lemmas \ref{lem:auxi}, \ref{lem:auxi2} and \ref{APP_1}, we are able to obtain the following estimates on the difference of  $Y_t$ and $Y_t^h$, which are defined in  Section \ref{subs:change}.

\begin{lemma}\label{dai5}

Let $\alpha \in (0,1]$.
Suppose that $b \in \mathcal{B}(\alpha)$ defined by \eqref{classB}.
\begin{itemize}
	\item[(i)]
	If $b$ is bounded, then for any $p \geq 2$, there exists a constant $C>0$ such that
	\begin{align*}
	\bE[|Y_T-Y_T^h|^p]  \leq Ch^{1/2}+Ch^{p\alpha/2}.
	\end{align*}
		
	\item[(ii)]
	If $b$ is of linear growth, then there exists a constant $C>0$ such that
	\begin{align*}
	\bE[|Y_T-Y_T^h|^2]  \leq Ch^{1/2}+Ch^{\alpha}
	\end{align*}
	
	\end{itemize}
\end{lemma}

\begin{proof}
Using Minkowski's inequality, we obtain
$\bE [| Y_T - Y^h_T|^p] \leq C\{S_1(p) + S_2(p)\}$
where 
\begin{align*}
& S_1(p) = \bE \Big[ \Big| \int_0^T \big\{ (\sm^{-1}b)_j(x_0 + \sm W_s) - (\sm^{-1}b)_j(x_0 + \sm W_{\eta_h(s)}) \big\}  dW^j_s  \Big|^p\Big],\\ 
& S_2(p)=   \bE \Big[ \Big| \int_0^T \big\{ |\sm^{-1}b(x_0 + \sm W_s)|^2 - |\sm^{-1}b(x_0 + \sm W_{\eta_h(s)})|^2 \big\} ds \Big|^p\Big].
\end{align*}
It follows from Burkholder-Davis-Gundy's inequality that,
\begin{align*}
S_1(p)
&\leq C \sum_{j=1}^d \int_0^T \bE \Big[ | (\sm^{-1}b)_j(x_0 + \sm W_s) - (\sm^{-1}b)_j(x_0 + \sm W_{\eta_h(s)}) |^p \Big] ds.
\end{align*}
Since $b$ is of linear growth,
$$\sum_{j=1}^d \int_0^h \bE \Big[ | (\sm^{-1}b)_j(x_0 + \sm W_s) - (\sm^{-1}b)_j(x_0 + \sm W_{\eta_h(s)}) |^p \Big] ds \leq Ch.$$ 
Furthermore, it follows from Proposition \ref{prop:A} ii) and Lemma \ref{lem:auxi} that 
\begin{align*}
\sum_{j=1}^d &\int_h^T \bE \Big[ | (\sm^{-1}b)_j(x_0 + \sm W_s) - (\sm^{-1}b)_j(x_0 + \sm W_{\eta_h(s)}) |^p \Big] ds \\
& \leq C\int_h^T \left\{ \frac{\sqrt{s-\eta_h(s)}}{\sqrt{\eta_h(s)}}+|s-\eta_n(s)|^{p\alpha} \right\} ds
\leq C(h^{1/2}+h^{p\alpha/2}).
\end{align*}
Therefore, $ S_1(p) \leq  C (h^{1/2}+h^{p\alpha}).$

Proof of (i).
We assume that $b$ is bounded.
Using H\"older's inequality, we obtain
\begin{align*}
S_2(p)
&\leq \bE \Big[  \int_0^T \Big| |\sm^{-1}b(x_0 + \sm W_s)|^2 - |\sm^{-1}b(x_0 + \sm W_{\eta_h(s)})|^2 \Big|^p ds \Big]\\
& \leq C \sum_{j=1}^d \int_0^T \bE \Big[ \Big| |(\sm^{-1}b)_j (x_0 + \sm W_s)|^2 - |(\sm^{-1}b)_j(x_0 + \sm W_{\eta_h(s)})|^2 \Big|^p \Big]ds.
\end{align*} 
Since $b$ is bounded, it holds that for any $x,y \in \br$ and $j=1,\ldots,d$,
\begin{align*}
\Big| |(\sm^{-1}b)_j (x)|^2 - |(\sm^{-1}b)_j(y)|^2 \Big|^p
\leq C\left\{|(\sm^{-1}b_A)_j (x)-(\sm_A^{-1}b)_j (y)|^p+|(\sm^{-1}b_H)_j (x)-(\sm_H^{-1}b)_j (y)|^p\right\}.
\end{align*}
Thus, by dividing the integral into two parts: from $0$ to $h$ and from $h$ to $T$, and applying a similar argument as above, we obtain
 \begin{align*}
S_2(p)
&\leq Ch + C\int_h^{T} \frac{\sqrt{s-\eta_h(s)}}{\sqrt{\eta_h(s)}}ds + Ch^{p\alpha/2}
\leq C(h^{1/2} + h^{p\alpha/2}).
\end{align*}
Thus, $\bE[|Y_T-Y_T^h|^p]  \leq Ch^{1/2}+Ch^{p\alpha/2}$.

Proof of (ii).
We assume that $b$ is of linear growth.
We observe that
\begin{align*}
S_2(2)
=\int_0^Tdu \int_{0}^{u} ds
&\bE \Big[ \big\{ |\sm^{-1}b(x_0 + \sm W_s)|^2 - |\sm^{-1}b(x_0 + \sm W_{\eta_h(s)})|^2\big\}  \\
&\quad \times \big\{ |\sm^{-1}b(x_0 + \sm W_u)|^2 - |\sm^{-1}b(x_0 + \sm W_{\eta_h(u)})|^2 \big\} \Big].
\end{align*}
Let $n $ be a natural number such that $(n-1)h< T \leq nh$.
Define $t_i^{h}=ih$ for $i=0,\ldots,n-1$ and $t_n^{h}=T$.
Since $b$ is of linear growth, we have
\begin{align*}
&S_2(2)\\
&\leq Ch
+\sum_{i=1}^{n-1} \int_{t_i^{h}}^{t_{i+1}^{h}}du \left\{ \int_{0}^{t_1^{h}}ds+\int_{t_1^{h}}^{t_i^{h}}ds + \int_{t_i^{h}}^{u}ds \right\}\\
&\quad \times
\bE \Big[ \big\{ |\sm^{-1}b(x_0 + \sm W_s)|^2 - |\sm^{-1}b(x_0 + \sm W_{\eta_h(s)})|^2\big\} 
\big\{ |\sm^{-1}b(x_0 + \sm W_u)|^2 - |\sm^{-1}b(x_0 + \sm W_{\eta_h(u)})|^2 \big\} \Big]\\
&\leq Ch
+\sum_{i=1}^{n-1} \int_{t_i^{h}}^{t_{i+1}^{h}}du \int_{t_1^{h}}^{t_i^{h}}ds\\
&\quad \times
\bE \Big[ \big\{ |\sm^{-1}b(x_0 + \sm W_s)|^2 - |\sm^{-1}b(x_0 + \sm W_{\eta_h(s)})|^2\big\} 
\big\{ |\sm^{-1}b(x_0 + \sm W_u)|^2 - |\sm^{-1}b(x_0 + \sm W_{\eta_h(u)})|^2 \big\} \Big].
\end{align*}
Since that for any $x,y \in \br^d$,
\begin{align*}
&|\sigma^{-1}b(x)-\sigma^{-1}b(y)|^2\\
&=\sum_{j=1}^{d}
\Big\{
|(\sigma^{-1}b_A)_j(x)|^2-|(\sigma^{-1}b_A)_j(y)|^2\\
&+\{(\sigma^{-1}b_H)_j(x)+\{(\sigma^{-1}b_H)_j(y)\} \{(\sigma^{-1}b_H)_j(x)-\{(\sigma^{-1}b_H)_j(y)\}\\
&+2(\sigma^{-1}b_H)_j(x) \{(\sigma^{-1}b_A)_j(x)-(\sigma^{-1}b_A)_j(y)\}
+2(\sigma^{-1}b_A)_j(y) \{(\sigma^{-1}b_H)_j(x)-(\sigma^{-1}b_H)_j(y)\}
\Big\}.
\end{align*}
by using Lemma \ref{lem:auxi2} with $t_1=\eta_h(s)$, $t_2=s$, $t_3=\eta_h(u)$ and $t_4=u$, we obtain	
\begin{align*}
S_2(2)
\leq Ch
+C\sum_{i=1}^{n-1} \int_{t_i^{h}}^{t_{i+1}^{h}}du \int_{t_1^{h}}^{t_i^{h}}ds
\left\{
\frac{h}{\sqrt{\eta_h(u)-s}\sqrt{\eta_h(s)}}
+\frac{h^{\frac{1}{2}+\frac{\alpha}{2}}}{\sqrt{\eta_h(s)}}
+\frac{h^{\frac{1}{2}+\frac{\alpha}{2}}}{\sqrt{\eta_h(u)-s}}
+h^{\alpha}
\right\}
\end{align*}
Therefore from Lemma \ref{APP_1}, we have
\begin{align*}
S_2(2)
\leq C(h+h^{\frac{1}{2}+\frac{\alpha}{2}}+h^{\alpha}).
\end{align*}
Thus, $\bE[|Y_T-Y_T^h|^2] \leq Ch^{1/2}+Ch^{\alpha}$.
\end{proof}

\subsection{Proof of Theorem \ref{main0}}

We first recall that $Z_t,Z_t^h,Y_t$ and $Y_t^h$ were defined in Section \ref{subs:change}.
It follows from Lemma \ref{dai5} that $Y^h_T$ converges in probability to $Y_T$ as $h \to 0$.
Thus  $Z^h_T$ also converges in probability to $Z_T$ as $h \to 0$.
Moreover $\bE[Z^h_T] = \bE[Z_T] = 1$ for all $h>0$.
Therefore, it follows from Proposition 4.12 in \cite{Ka} that 
\begin{equation} \label{eq:ghZ}
\lim_{h \to 0} \bE[|Z^h_T - Z_T|] = 0.
\end{equation}  
On the other hand, since $f$ is bounded, it follows from (\ref{eq:error}) that
$$|\bE[f(X)-f(X^h)]| \leq C\bE [|Z^h_T - Z_T|].$$
This estimate together with (\ref{eq:ghZ}) implies the desired result.
\qed

\subsection{Proof of Theorem \ref{main1}} \label{proof:m1}
It is clear that  $|e^x - e^y| \leq (e^x+ e^y)|x-y|$. This estimate and H\"older's inequality imply that $|\bE[f(X)-f(X^h)]|$ is bounded by 
\begin{align*}
&  \bE \big[\big |f(x_0+\sm W)(Z_T + Z^h_T)(Y_T - Y^h_T)\big|\big]\\ 
&\leq \big\| f(x_0 +\sm W)(Z_T + Z^h_T)\big\|_2 \| Y_T - Y^h_T\|_2\\ 
&\leq \Big(\bE\big[|f(x_0+\sm W)|^r\big]\Big)^{2/r} \Big( \bE\big[ |Z_T + Z_T^h|^{2r/(r-2)}\big]\Big)^{(r-2)/r} \| Y_T - Y^h_T\|_2.
\end{align*}
Thanks to the integrability condition of $f$ and Lemma \ref{lem:powZ}, 
\begin{equation*}
 \Big(\bE\big[|f(x_0+\sm W)|^r\big]\Big)^{2/r} \Big( \bE\big[ |Z_T + Z_T^h|^{2r/(r-2)}\big]\Big)^{(r-2)/r} \leq C < \infty.
 \end{equation*}
This together with Lemma \ref{dai5}  implies the desired result.
\qed

\subsection{Proof of Corollary \ref{main1.1}}
Since $g$ is $\beta$-H\"older continuous, it holds that $\bE\left[ |g\left( \max_{0\leq s \leq T}|x_0+\sigma W_s| \right)|^r \right] <\infty$ for any $r >2$.
Thanks to Theorem \ref{main1}, it remains to estimate
\begin{align*}
\left| \bE\left[ g\left( \max_{0\leq s \leq T}|X_s^h| \right) \right] - \bE\left[g\left(\max_{0\leq s \leq T}|X_{\eta_h(s)}^{h}| \right)\right] \right|.
\end{align*}
Since $g$ is $\beta$-H\"older continuous and $b$ is sub-linear growth, we have
\begin{align}\label{max_1}
&\left| \bE\left[ g\left( \max_{0\leq s \leq T}|X_s^h| \right) \right] - \bE\left[g\left(\max_{0\leq s \leq T}|X_{\eta_h(s)}^{h}| \right)\right] \right| \notag\\
&\leq C \bE\left[\left| \max_{0\leq s \leq T}|X_s^h| - \max_{0\leq s \leq T}|X_{\eta_h(s)}^{h}| \right|^{\beta}\right]
\leq C \bE\left[ \max_{0\leq s \leq T}|X_s^h - X_{\eta_h(s)}^{h}| \right] ^{\beta}\notag\\
&\leq C \bE\left[ \max_{0\leq s \leq T} \{|b(X_{\eta_h(s)}^h)|(s-\eta_h(s)) + |\sigma(W_s-W_{\eta_h(s)})|\} \right]^{\beta} \notag\\
&\leq C h^{\beta} 
+C\bE \left[\max_{0\leq s \leq T}|W_s-W_{\eta_h(s)}|\} \right]^{\beta}.
\end{align}
By modulus continuity of Brownian motion (e.g. Lemma 4.4 in \cite{Pe95}), we have
\begin{align*}
\bE\left[\max_{0\leq s \leq t \leq T, |t-s|\leq h} |W_t-W_s|^2 \right]
\leq C h \log\left(1/h\right),
\end{align*}
Thus from \eqref{max_1}, we obtain
\begin{align*}
\left| \bE\left[ g\left( \max_{0\leq s \leq T}|X_s^h| \right) \right] - \bE\left[g\left(\max_{0\leq s \leq T}|X_{\eta_h(s)}^{h}| \right)\right] \right|
\leq C\left\{h^{\beta}+(h \log(1/h))^{\beta/2} \right\},
\end{align*}
which implies the proof of the statement.
\qed

\subsection{Proof of Corollary \ref{main1.5}}
We first note that if $b$ is bounded, then it holds from Theorem 2.1 in \cite{LeMe} (see also Corollary 3.2 in \cite{NT}) that there exists a density function $p_t^h$ of $X_t^h$ for $t \in (0,T]$ and it satisfies the following Gaussian upper bound, i.e., 
\begin{align*}
p_t^h(x) \leq C \frac{e^{-\frac{|x-x_0|^2}{2ct}}}{t^{d/2}}.
\end{align*}
for some positive constants $C$ and $c$.

Now we prove that $\bE\left[\left| f\left( \int_0^T g(x_0+W_s) ds \right) \right|^r \right]$ is finite for any $r>2$.
Since $|g(x)|\leq Ke^{K|x|}$, it follows from Jensen's inequality that for any $r>2$,
\begin{align*}
\bE\left[\left| f\left( \int_0^T g(x_0+W_s) ds \right) \right|^r \right]
\leq C+C \int_0^T \bE[|g(x+W_s)|^r]ds
\leq C+C \int_0^T \prod_{i=1}^d \bE[e^{rKW_s^i}] ds.
\end{align*}
Since $x^2/(4s)+K^2r^2s \geq Krx$, we have
\begin{align*}
\bE\left[\left| f\left( \int_0^T g(x_0+W_s) ds \right) \right|^r \right]
&\leq C+C \int_0^T e^{dKrs} ds<\infty.
\end{align*}
Thanks to Theorem \ref{main1}, it remains to prove that
\begin{align*}
\left| \bE\left[ f \left( \int_0^Tg(X_s^h)ds \right) \right] - \bE\left[f\left(\int_0^Tg(X_{\eta_h(s)}^h)ds \right)\right] \right| \leq C h^{\beta/2}.
\end{align*}
Since $f$ is a Lipschitz continuous function, we have
\begin{align*}
\left| \bE\left[ f\left( \int_0^Tg(X_s^h)ds \right) \right] - \bE\left[f\left(\int_0^Tg(X_{\eta_h(s)}^h)ds \right)\right] \right|
&\leq C\int_0^T \bE\left[\left|g(X_s^h)-g(X_{\eta_h(s)}^h) \right| \right] ds
\end{align*}
If $s \in (0,h]$, then by using the Gaussian upper bound for $p_s^h(x)$, we have
\begin{align*}
\int_0^h  \bE\left[\left|g(X_s^h)-g(X_{\eta_h(s)}^h) \right|\right] ds
&\leq \int_0^h \bE[|g(X_s^h)|] ds + |g(x_0)| h\\
&\leq C \int_0^h ds \int_{\br^d} |g(x)| \frac{e^{-\frac{|x-x_0|^2}{2cs}}}{s^{d/2}} + |g(x_0)| h\\
&\leq C h.
\end{align*}
On the other hand, for $s \in [h,T]$, using the Gaussian upper bound for $p_{\eta_h(s)}^h$ and following the proof of Lemma \ref{lem:auxi} (see also Lemma 3.5 in \cite{NT}), we have
\begin{align*}
\bE[|g(X_s^h) - g(X_{\eta_h(s)}^h)|] \leq \frac{C\sqrt{s-\eta_h(s)}}{\sqrt{\eta_h(s)}} + C h^{\beta/2}.
\end{align*}
Therefore, we conclude the proof of the statement.
\qed

\subsection{Proof of Corollary \ref{main2}}
Recall that  $\tau_D^W,\tau_D^{W,h},Z,Z^h$ were defined at the beginning of Section \ref{subs:change}. 
It suffices to proof the statement for the case that $g$ is positive. It follows from \eqref{eq:error2} that $\bE[g(X_T)\1_{(\tau_D > T)}] - \bE [g(X^h_T)\1_{(\tau^h_D > T)}] = E_1 + E_2$ where 
\begin{align*} 
E_1 & =  \bE[g(x_0 + \sm W_T )(Z_T - Z^h_T) \1_{(\tau^{W,h}_D > T)}],\\
E_2 &= \bE[ g(x_0 + \sm W_T )Z_T (\1_{(\tau^W_D > T)} - \1_{(\tau^{W,h}_D > T)})]. 
\end{align*}
It follows from the proof of Theorem \ref{main1} that 
\begin{equation} \label{dai2}
 |E_1| \leq \bE[| g(x_0 + \sm W_T )(Z_T - Z^h_T)|] \leq Ch^{\frac{\alpha}{2} \wedge \frac14}.
 \end{equation}
 Applying H\"older's inequality, we have 
 \begin{align*}
 |E_2| \leq \|Z_T\|_q \| g(x_0 + \sm W_T ) (\1_{(\tau^W_D > T)} - \1_{(\tau^{W,h}_D > T)})\|_p,
  \end{align*}
where $q$ is the conjugate of $p$. Thanks to Lemma \ref{lem:powZ} and the fact  $\tau^{W,h}_D  \geq \tau^W_D$, we have
 \begin{align*}
 |E_2| &\leq C_p \Big( \bE \Big[ g^p(x_0 + \sm W_T ) \1_{(\tau^{W,h}_D \geq T)}\Big] -  \bE \Big[ g^p(x_0 + \sm W_T ) \1_{(\tau^{W}_D \geq T)}\Big] \Big)^{1/p}.
\end{align*}
It follows from  Theorem 2.4 in \cite{G00} that there exists a constant $K(T)$ such that 
  \begin{align*}
 |E_2| & \leq C_p \Big(\frac{K(T)\|g^p\|_\infty}{1\wedge \eps^4}\Big)^{1/p}h^{\frac{1}{2p}}.
  \end{align*}
Combining this estimate with \eqref{dai2} completes the proof.
\qed

\subsection{Proof of Theorem \ref{main2.5}}
The proof is the based on the perturbation or Levi's parametrix method (see \cite{Fr64}) for the density functions $p_t(x_0,y)$ and $p_t^h(x_0,y)$.
It is known that when the drift coefficient $b$ is bounded and H\"older continuous, it holds that for any $t\in(0,T]$
\begin{align}\label{rep:density_1}
p_t(x_0,y)
&=g_{ta}(x_0,y)+ \int_{0}^{t}ds \int_{\br^d} dz \left\langle \nabla_x g_{(t-s)a}(z,y), b(z) \right\rangle p_t(x_0,z)\notag\\
&=g_{ta}(x_0,y)+\int_{0}^{t} \bE\left[ \left\langle \nabla_x g_{(t-s)a}(X_s,y), b(X_s) \right\rangle \right] ds,
\end{align}
where $a=\sigma \sigma^*$ and $g_{ta}(x_0,y)$ is called the parametrix (see, Chapter 1 in \cite{Fr64}).

We first consider a  representation which is similar to  \eqref{rep:density_1} for the density functions $p_t(x_0,y)$ and $p_t^h(x_0,y)$ under the assumption that the drift coefficient is bounded measurable.
Recently, Makhlouf \cite[Theorem 3.1]{Ma15} proved that the representation \eqref{rep:density_1} also holds for a Brownian motion with random drift $b=(b_t)_{0\leq t \leq T}$ under the suitable growth condition.
For the convenience of the reader, we  give a proof below.
\begin{proposition}[Makhlouf \cite{Ma15}]\label{rep:density_3}
	Let $W=(W_t)_{0\leq t \leq T}$ be a $d$-dimensional $(\mathcal{F}_t)_{0\leq t \leq T}$-Brownian motion.
	Suppose that the stochastic process $b=(b_t)_{0\leq t \leq T}$ is adapted to $(\mathcal{F}_t)_{0\leq t \leq T}$ and there exists a constant $K>0$ such that $\sup_{0\leq t \leq T}|b_t|\leq K$ almost surely.
	Then, for any $t \in (0,T]$ and $x \in \br^d	$, the stochastic process $Y_t:=x+\int_{0}^{t} b_s ds + \sigma W_t$ admits a density $\gamma_t(x,\cdot)$ with respect to Lebesgue measure and for any $y \in \br^d$,
	\begin{align}\label{rep:density_2}
	\gamma_t(x,y)
	&=g_{ta}(x,y)+ \int_{0}^{t} \bE\left[ \left\langle \nabla_x g_{(t-s)a}(Y_s,y), b_s \right\rangle \right] ds.
	\end{align}
\end{proposition}
\begin{proof}
	It suffices to prove that for any  infinitely differentiable functions $f:\br^d \to \br$ with compact support  
	\begin{align*}
	\bE[f(Y_t)]
	=\int_{\br^d} f(y) g_{ta}(x,y) dy
	+ \int_{\br^d} dy f(y) \int_0^t ds \bE\left[ \langle \nabla_x g_{(t-s)a}(Y_s,y), b_s \rangle \right].
	\end{align*}
	It is well-known that the function $u(s,x):=\bE\left[f\left( x+\sigma W_{t-s} \right) \right]$
	is a solution to the following partial differential equation:
	\begin{align}\label{heat_eq_1}
	\partial_s u(s,x)
	+\frac{1}{2}\sum_{i,j=1}^d a_{i,j} \frac{\partial^2 }{\partial x_i \partial x_j}u(s,x)
	&=0,~(s,x)\in [0,t) \times \br^d,\\
	u(t,x)&=f(x),~ x\in \br^d \notag.
	\end{align}
	Hence we have
	\begin{align}
	\bE\left[f\left(x+\sigma W_t \right) \right]
	&=u(0,x), \label{expectation_2}\\
	\bE[f(Y_t)]
	&=\bE\left[u(t,Y_t)\right].\label{expectation_3}
	\end{align}
	By using It\^o's formula and \eqref{heat_eq_1},	it holds that for any $\varepsilon \in (0,T)$,
	\begin{align}\label{Ito_3}
	u(t-\varepsilon,Y_{t-\varepsilon})
	&=u(0,x)+\int_0^{t-\varepsilon} \langle \nabla_x u(s,Y_s), b_s \rangle ds
	+\sum_{i,j=1}^{d} \int_0^{t-\varepsilon} \sigma_{i,j} \frac{\partial}{\partial x_i} u(s,Y_s) dW_s^j.
	\end{align}
	Since for any $i=1,\ldots,d$, $s \in [0,T)$ and $x \in \br^d$,
	\begin{align}\label{deriva_u}
	\frac{\partial}{\partial x_i} u(s,x)
	=\frac{\partial}{\partial x_i} \bE\left[f\left(x+\sigma W_{t-s} \right) \right]
	=\int_{\br^d} f(y) \frac{\partial}{\partial x_i} g_{(t-s)a}(x,y)dy
	\end{align}
	and
	\begin{align*}
	\left| \frac{\partial}{\partial x_i} u(s,x)\right|
	\leq \frac{C\|f\|_{\infty}}{(t-s)^{1/2}},
	\end{align*}
	for some constant $C>0$, the stochastic integral in \eqref{Ito_3} is martingale.
	By taking the expectation and Fubini's theorem, we have from \eqref{expectation_2}
	\begin{align*}
	\bE[u(t-\varepsilon,Y_{t-\varepsilon})] 
	&=\bE\left[f\left(x+\sigma W_t \right) \right]
	+\int_0^{t-\varepsilon} \bE\left[ \langle \nabla_{x} u(s,Y_s),b_s \rangle \right]ds.
	\end{align*}
	Taking $\varepsilon \to 0$ and using the dominated convergence theorem, we have from \eqref{expectation_3} and \eqref{deriva_u},
	\begin{align*}
	\bE[f(Y_t)]
	&=\lim_{\varepsilon \to 0+} \bE[u(t-\varepsilon,Y_{t-\varepsilon})]
	= \bE\left[f\left(x+\sigma W_t \right) \right]
	+ \int_0^{t} \bE\left[ \langle \nabla_x u(s,Y_s),b_s\rangle \right]ds \notag\\ 
	&= \int_{\br^d} f(y) g_{ta}(x,y) dy
	+ \int_{\br^d}  f(y) \int_0^t  \bE\left[ \langle \nabla_x g_{(t-s)a}(Y_s,y), b_s \rangle \right]dsdy. 
	\end{align*}
	This concludes the proof.	
\end{proof}

\begin{proof}[Proof of Theorem \ref{main2.5}]
	Using Proposition \ref{rep:density_3}, we have
	\begin{align*}
	p_T(x_0,y)
	&=g_{Ta}(x_0,y)+ \int_{0}^{T} \bE\left[ \left\langle \nabla_x g_{(T-s)a}(X_s,y), b(X_s) \right\rangle \right] ds,\\
	p_T^h(x_0,y)
	&=g_{Ta}(x_0,y)+ \int_{0}^{T} \bE\left[ \left\langle \nabla_x g_{(T-s)a}(X_s^h,y), b(X_{\eta_h(s)}^h) \right\rangle \right] ds.
	\end{align*}
	Moreover, from Lemma \ref{lem:decom}, we have
	\begin{align*}
	p_T(x_0,y)-p_T^h(x_0,y)
	&=\int_{0}^{T} \bE\left[ \left\langle \nabla_x g_{(T-s)a}(x_0+\sigma W_s,y), Z_Tb(x_0+\sigma W_s)-Z_T^h b(x_0+\sigma W_{\eta_h(s)}) \right\rangle \right] ds.
	\end{align*}
	By using Jensen's inequality and Schwarz's inequality,  there exists $c>0$ such that
	\begin{align*}
	\left| p_T(x_0,y)-p_T^h(x_0,y) \right|
	&\leq \int_{0}^{T} \bE\left[ \left| \nabla_x g_{(T-s)a}(x_0+W_s,y) \right| \left| Z_Tb(x_0+\sigma W_s)-Z_T^h b(x_0+\sigma W_{\eta_h(s)}) \right| \right] ds\\
	&\leq C \int_{0}^{T} \frac{1}{\sqrt{T-s}} \bE\left[ g_{c(T-s)}(x_0+W_s,y) \left| Z_Tb(x_0+\sigma W_s)-Z_T^h b(x_0+\sigma W_{\eta_h(s)}) \right| \right] ds\\
	&\leq C \int_{0}^{T} \frac{1}{\sqrt{T-s}} (A_s+B_s) ds,
	\end{align*}
	where
	\begin{align*}
	A_s
	&:=\bE\left[ g_{c(T-s)}(x_0+\sigma W_s,y) \left| Z_T-Z_T^h \right| \right]\\
	B_s
	&:=\bE\left[ g_{c(T-s)}(x_0+\sigma W_s,y) \left| Z_T \right| \left| b(x_0+\sigma W_s)-b(x_0+\sigma W_{\eta_h(s)}) \right| \right].
	\end{align*}
	Note that for any $q>1$, there exist $C_q$ and $c_q$ such that 
	\begin{align}\label{esti_gauss}
	\bE\left[|g_{c(T-s)}(x_0+\sigma W_s,y)|^q\right]^{1/q}
	\leq C_q \left( \frac{T}{T-s} \right)^{\frac{d(q-1)}{2q}} g_{c_qT}(x_0,y).
	\end{align}
	
	We first consider the upper bounded for $A_s$.
	By using $|e^x-e^y|\leq (e^x+e^y)|x-y|$ and H\"older's inequality, for any $p,q>1$ with $1/p+1/q=1$ and $r>1$, we have
	\begin{align*}
	A_s
	&\leq \bE\left[ |g_{c(T-s)}(x_0+\sigma W_s,y)|^q\right]^{1/q} \bE\left[ |Z_T+Z_T^h|^{p} \left| Y_T-Y_T^h \right|^p \right]^{1/p}\\
	&\leq C_q \left( \frac{T}{T-s} \right)^{\frac{d(q-1)}{2q}} g_{c_qT}(x_0,y) 
	\bE\left[ |Z_T+Z_T^h|^{pr/(r-1)} \right]^{(r-1)/(pr)} \bE\left[ \left| Y_T-Y_T^h \right|^{pr} \right]^{1/(pr)}.
	\end{align*}
	Therefore, it holds that
	\begin{align}\label{esti_As_1}
	\int_{0}^{T} \frac{1}{\sqrt{T-s}} A_s ds
	&\leq C_{p,r} C_q g_{c_qT}(x_0,y)
	\int_{0}^{T} \frac{1}{\sqrt{T-s}} \left( \frac{T}{T-s} \right)^{\frac{d(q-1)}{2q}} ds
	\bE\left[ \left| Y_T-Y_T^h \right|^{pr} \right]^{1/(pr)}.
	\end{align}
	By choosing $p>d$ that is $q=p/(p-1)<d/(d-1)$, from Lemma \ref{dai5}, we have
	\begin{align}\label{esti_As_3}
	\int_{0}^{T} \frac{1}{\sqrt{T-s}} A_s ds
	&\leq C_{p,r} C_q g_{c_qT}(x_0,y) T^{1/2}
	\bE\left[ \left| Y_T-Y_T^h \right|^{pr} \right]^{1/(pr)} \notag\\
	&\leq C_{p,q,r} g_{c_qT}(x_0,y) T^{1/2}
	\{h^{\alpha/2}+h^{1/(2pr)}\}.
	\end{align}
	
	Now we give an upper bound for $B_s$.
	By using H\"older's inequality, for any $p>d$ and $q=p/(p-1)$ and $r>1$, we have
	\begin{align*}
	B_s
	&\leq \bE\left[ |g_{c(T-s)}(x_0+\sigma W_s,y)|^q\right]^{1/q} \bE\left[ |Z_T|^{pr/(r-1)}\right]^{(r-1)/(pr)} \bE\left[ \left| b(x_0+\sigma W_s)-b(x_0+\sigma W_{\eta_h(s)}) \right|^{pr} \right]^{1/(pr)}\\
	&\leq C_{p,q,r} \left( \frac{T}{T-s} \right)^{\frac{d(q-1)}{2q}} g_{c_qT}(x_0,y)
	\bE\left[ \left| b(x_0+\sigma W_s)-b(x_0+\sigma W_{\eta_h(s)}) \right|^{pr} \right]^{1/(pr)}.
	\end{align*}
	By Lemma \ref{lem:auxi} for any $s\geq h$,
	\begin{align*}
	\bE\left[ \left| b(x_0+\sigma W_s)-b(x_0+\sigma W_{\eta_h(s)}) \right|^{pr} \right]^{1/(pr)}
	\leq C\{h^{\alpha/2}+\frac{h^{1/(2pr)} }{\eta_h(s)^{1/(2pr)}}\}.
	\end{align*}
	Since $b$ is bounded, we have
	\begin{align}\label{esti_Bs_1}
	\int_{0}^{T} \frac{1}{\sqrt{T-s}} B_s ds
	&\leq C_{p,q,r}g_{c_qT}(x_0,y)
	\int_{0}^{h} \frac{1}{\sqrt{T-s}} \left( \frac{T}{T-s} \right)^{\frac{d(q-1)}{2q}} ds \notag\\
	&+C_{p,q,r}g_{c_qT}(x_0,y) h^{\alpha/2}
	\int_{h}^{T} \frac{1}{\sqrt{T-s}} \left( \frac{T}{T-s} \right)^{\frac{d(q-1)}{2q}}ds \notag\\
	&+C_{p,q,r}g_{c_qT}(x_0,y) h^{1/(2pr)}
	\int_{h}^{T} \frac{1}{\sqrt{T-s}} \left( \frac{T}{T-s} \right)^{\frac{d(q-1)}{2q}} \frac{1}{\eta_h(s)^{1/(2pr)}} ds.
	\end{align}
	By the assumption $h \in (0,T/2)$, it holds that
	\begin{align*}
	\int_{0}^{h} \frac{1}{\sqrt{T-s}} \left( \frac{T}{T-s} \right)^{\frac{d(q-1)}{2q}} ds
	\leq 2^{\frac{1}{2}+\frac{d(q-1)}{2q}} \frac{h}{T^{1/2}}.
	\end{align*}
	Since $q<d/(d-1)$, we have
	\begin{align*}
	\int_{h}^{T} \frac{1}{\sqrt{T-s}} \left( \frac{1}{T-s} \right)^{\frac{d(q-1)}{2q}} ds
	=\frac{2q}{q-dq+d}(T-h)^{\frac{q-dq+d}{2q}}
	\end{align*}
	and
	\begin{align*}
	\int_{h}^{T} \frac{1}{\sqrt{T-s}} \left( \frac{1}{T-s} \right)^{\frac{d(q-1)}{2q}} \frac{1}{\eta_h(s)^{1/(2pr)}} ds
	\leq \left(T-h\right)^{1-1/(2pr)} B\left(\frac{1}{2}+\frac{d(q-1)}{2q},\frac{1}{2pr}\right),
	\end{align*}
	where $B(x,y)$ is the beta function.
	Therefore, we obtain
	\begin{align}\label{esti_Bs_3}
	\int_{0}^{T} \frac{1}{\sqrt{T-s}} B_s ds
	\leq C_{p,q,r}g_{c_pT}(x_0,y) \left\{\frac{h}{T^{1/2}}+h^{\alpha/2}+h^{1/(2pr)} \right\},
	\end{align}
	which concludes the proof.

\end{proof}

\subsection{Proof of Theorem \ref{main3}}
In the same way as in subsection \ref{subs:change}, we have the following Lemma.
\begin{lemma}\label{change_RSDE}
Let $X$ be a solution of the SDE \eqref{RSDE} and let $X^h$ be the Euler-Maruyama approximation of $X$ defined in \eqref{EM_RSDE} with $h>0$.
Suppose that  $b$ is a measurable function with sub-linear growth. Let $U=(U_t)_{0 \leq t \leq T}$ be the unique solution of the equation $U_t=x_0 + \sigma W_t + L^0_t(U)$ and
\begin{align*}
\hat{Z}_t&:=e^{\hat{Y_t}}, \quad  \hat{Y}_t:=\int_0^t b(U_s) dW_s - \frac12 \int_0^t b^2(U_s) ds \\
\hat{Z}_t^h&:=e^{\hat{Y}_t^h}, \quad  \hat{Y}^h_t:=\int_0^t b(U_{\eta_h(s)}) dW_s - \frac12 \int_0^t b^2(U_{\eta_h(s)}) ds.
\end{align*}
Then it holds that 
\begin{align*}
\bE[f(X)] - \bE[f(X^h)] = \bE[f(U)(\hat{Z}_T-\hat{Z}_T^h)]
\end{align*}
for all measurable functions $f: C[0,T] \to \br$ provided that the above expectations are integrable.
\end{lemma}
\begin{proof}
We define new measures $\hat{\bQ}$ and $\hat{\bQ}^h$ as 
\begin{align*}
&\frac{d \hat{\bQ}}{d \pr} =  \exp \Big( -\int_0^T \sm^{-1} b(X_s) dW_s - \frac{1}{2} \int_0^T |\sm^{-1}b(X_s)|^2 ds \Big),\\
&\frac{d \hat{\bQ}^h}{d \pr	} =  \exp \Big( -\int_0^T \sm^{-1} b(X^h_{\eta_h(s)}) dW_s - \frac{1}{2} \int_0^T |\sm^{-1}b(X^h_{\eta_h(s)})|^2 ds\Big).
\end{align*}
Since $b$ is of sub-linear growth and the fact that $0 \leq L^0_t(X) \leq |\sigma| \sup_{0 \leq s \leq t} |W_s|$, by following the proof of Lemma \ref{lem:decom} we can show that $\hat{\bQ}$ and $\hat{\bQ}^h$ are probability measures.
Furthermore, it follows from the Girsanov theorem that the processes $\hat{B}=(\hat{B}_t)_{0\leq t \leq T}$ and $\hat{B}^h=(\hat{B}_t^h)_{0\leq t \leq T}$ defined by
\begin{align*}
\hat{B}_t = W_t + \int_0^t \sm^{-1}b(X_s)ds, \ \hat{B}^{h}_t = W_t + \int_0^t \sm^{-1}b(X^h_{\eta_h(s)})ds, 0\leq t \leq T,
\end{align*}
are Brownian motions with respect to $\hat{\bQ}$ and $\hat{\bQ}^h$ respectively.
Note that $X_s=x_0 + \sigma \hat{B}_s + L_s^0(X)$ and $X_s^h = x_0 + \sigma \hat{B}_s^h + L_s^0(X^h)$.
Therefore,
\begin{align*}
&\bE[f(X)] = \bE_{\hat{\bQ}} \Big[f(X)\frac{d\pr}{d\hat{\bQ}}\Big] \\
&= \bE_{\hat{\bQ}} \Big[ f(X) \exp \Big( \int_0^T \sigma^{-1}b(X_s)d\hat{B}_s - \frac12 \int_0^T |\sigma^{-1}b(X_s)|^2ds\Big)\Big]\\
&= \bE_{\hat{\bQ}} \Big[ f(x_0+\sigma \hat{B} + L^0(X)) \exp \Big( \int_0^T \sigma^{-1}b(x_0+\sigma \hat{B}_s + L_s^0(X))d\hat{B}_s - \frac12 \int_0^T |\sigma^{-1}b(x_0+\sigma \hat{B}_s + L_s^0(X))|^2ds\Big)\Big].
\end{align*}
Since $(X, \hat{B})|_{\hat\bQ} \overset{d}{=} (U, W)|_{\pr}$, the above term equals to
\begin{align*}
& \bE \Big[ f(x_0+\sigma W + L^0(U)) \exp \Big( \int_0^T \sigma^{-1}b(x_0+\sigma W_s + L_s^0(U))dW_s - \frac12 \int_0^T |\sigma^{-1}b(x_0+\sigma W_s + L_s^0(U))|^2ds\Big)\Big]\\
&= \bE \Big[ f(U) \exp \Big( \int_0^T \sigma^{-1}b(U_s)dW_s - \frac12 \int_0^T |\sigma^{-1}b(U_s)|^2ds\Big)\Big]\\
&= \bE [f(U)\hat{Z}_T].
\end{align*}
Repeating the previous argument leads to $\bE[f(X^h)] = \bE [f(U)\hat{Z}^h_T],$ which concludes the statement.
\end{proof}

In the same way as Lemma \ref{lem:powZ}, we have the following estimate of the moments of $\hat{Z}$ and $\hat{Z}^h$.

\begin{lemma}\label{Lp-bound}
Suppose that $b$ is of sub-linear growth.
Then for any $p>0$,
$$\bE[|\hat{Z}_T|^p + |\hat{Z}^h_T|^p] \leq C < \infty,$$
for some constant $C$ which is not depend on $h$.
\end{lemma}

Finally, we introduce the following auxiliary estimate.
\begin{lemma}\label{auxiliary_RSDE}
Let $U$ as in Lemma \ref{change_RSDE}.
Suppose that $\zeta$ is $\alpha$-H\"older continuous with $\alpha \in (0,1]$, then for any   $t>s>0$,
\begin{equation*}
\bE [|\zeta(U_t) - \zeta(U_s)|^p]
\leq C_p (t-s)^{p\alpha/2}.
\end{equation*}
\end{lemma}
\begin{proof}
By H\"older continuity of $\zeta$, we have
\begin{equation*}
\bE [|\zeta(U_t) - \zeta(U_s)|^p)]
\leq \bE [|U_t - U_s|^{\alpha p}]
\leq C\bE [|W_t - W_s|^{p\alpha }]+C\bE [|L_t^0(U) - L_s^0(U)|^{p\alpha }].
\end{equation*}
Hence it is sufficient to prove that 
\begin{equation*}
\bE [|L_t^0(U) - L_s^0(U)|^{p\alpha}]  \leq C_p (t-s)^{p\alpha/2}.
\end{equation*}
Using Lemma \ref{Skorohod}, we have
\begin{align*}
L_s^0(U) \leq L_t^0(U)
&\leq L_s^0(U) + \sup_{s \leq u \leq t} \max\left(0, - \sigma (W_u-W_s) \right).
\end{align*}
Therefore, $|L_t^0(U) - L_s^0(U)| 
\leq  |\sigma| \sup_{s \leq u \leq t} |W_u-W_s|$.
Hence applying Burkholder-Davis-Gundy's inequality, we have 
\begin{align*}
\bE [|L_t^0(U) - L_s^0(U)|^{p\alpha}]  
&\leq C \bE[\sup_{s \leq u \leq t} |W_u-W_s|^{p\alpha}]
\leq C (t-s)^{p\alpha/2}.
\end{align*}
This concludes the proof.

\end{proof}

\noindent \emph{Proof of Theorem \ref{main3}}
We first recall that $\hat{Z}_t,\hat{Z}_t^h,\hat{Y}_t$ and $\hat{Y}_t^h$ are defined in Lemma \ref{change_RSDE}.
Using Lemmas \ref{change_RSDE}  and the elementary estimate $|e^x-e^y|\leq (e^x+e^y)|x-y|$, we have
$$|\bE[f(X)]-\bE[f(X^h)]| \leq \bE\Big[ \big|f(U) (\hat{Z}_T + \hat{Z}^h_T)(\hat{Y}_T-\hat{Y}_T^h)\big|\Big].$$
Thanks to Lemma \ref{Lp-bound} and the H\"older's inequality, for some $r>2$, we have
\begin{align*}
|\bE[f(X)]-\bE[f(X^h)]|&\leq C \bE[|f(U)|^r]^{2/r} ||\hat{Y}_T-\hat{Y}_T^h||_2.
\end{align*}
By a similar argument as the proof of Lemma \ref{dai5}, we can show that 
\begin{align*}
||\hat{Y}_T-\hat{Y}_T^h||_2 \leq C h^{\alpha/2},
\end{align*}
which concludes the proof.
\qed

\begin{remark}
The conclusion of Theorem \ref{main3} still holds if we relax the condition $f$ bounded to 
$\bE[|f(U)|^r]< \infty$ for some $r >2$.
\end{remark}


\section*{Acknowledgements}
This work was completed while  H.N. was staying at Vietnam Institute for Advanced Study in Mathematics (VIASM).
He would like to thank the institute for support.
This work was also partly supported  by the Vietnamese National Foundation for Science and Technology Development (NAFOSTED) under Grant Number 101.03-2014.14.
D.T. was supported by JSPS KAKENHI Grant Number 16J00894.
Both authors thank Arturo Kohatsu-Higa and an anonymous  referee for their helpful comments and  suggestions which have improved the readability of this paper.



\end{document}